\documentclass[reqno, a4paper, 11pt]{article}
\usepackage[T1]{fontenc}
\usepackage[utf8]{inputenc}
\usepackage[english]{babel}
\usepackage{amsmath}
\usepackage{amssymb}
\usepackage{mathrsfs}
\usepackage{amsfonts}
\usepackage{amsthm}
\usepackage{bm}
\usepackage{enumerate}
\usepackage[shortlabels]{enumitem}
\usepackage{moreenum}
\usepackage{mathtools}
\usepackage{xcolor}
\usepackage{csquotes}
\usepackage{hyperref}
\usepackage{subfigure}

\numberwithin{equation}{section}

\newtheorem{theorem}{Theorem}[section]
\newtheorem{definition}[theorem]{Definition}

\newtheorem{lemma}[theorem]{Lemma}
\newtheorem{corollary}[theorem]{Corollary}
\newtheorem{proposition}[theorem]{Proposition}

\newtheorem{remark}[theorem]{Remark}

\newtheorem{assumption}[theorem]{Assumption}

\newcommand{\ii}{i}
\newcommand{\im}{\operatorname{Im}}
\newcommand{\re}{\operatorname{Re}}
\newcommand{\sgn}{\operatorname{sgn}}

\newcommand{\ma}{\begin{pmatrix}}
	\newcommand{\am}{\end{pmatrix}}
\newcommand{\qqq}{\qquad \qquad}
\newcommand{\qq}{\quad \quad}
\def\supp{\mathop{\rm supp}\nolimits}
\def\Im{{\rm Im\,}}

\begin{document}

	\title{Direct resonance problem for Rayleigh seismic surface waves}
	\author{Samuele Sottile\footnotemark[1]\;\,\footnotemark[2]}
	\renewcommand{\thefootnote}{\fnsymbol{footnote}}
		\footnotetext[1]{Centre for Mathematical Sciences, Lund University, 221 00 Lund, Sweden; \tt\href{mailto:samuele.sottile@math.lu.se}{samuele.sottile@math.lu.se}}
			\footnotetext[2]{Department of Materials Science and Applied Mathematics, Malm\"{o} University, SE-205 06 Malm\"{o}, Sweden; \tt\href{mailto:samuele.sottile@mau.se}{samuele.sottile@mau.se}}
	
	\maketitle
	
	\begin{abstract}
		In this paper we study the direct resonance problem for Rayleigh seismic surface waves and obtain a constraint on the location of resonances and establish a forbidden domain as the main result. In order to obtain the main result we make a Pekeris-Markushevich transformation of the Rayleigh system with free surface boundary condition such that we get a matrix Schr{\"o}dinger-type form of it. We obtain parity and analytical properties of its fundamental solutions, which are needed to prove the main theorem. We construct a function made up by Rayleigh determinants factors, which is proven to be entire, of exponential type and in the Cartwright class and leads to the constraint on the location of resonances.\\[4mm]
	\end{abstract}

\vskip 0.25cm

\section{Introduction.}\label{intro}
%

The Rayleigh boundary value problem for the vertically inhomogeneous elastic isotropic medium in the half-space follows the equations (see \cite{4authorsPaper}):
	\begin{equation}\label{Rayleigh1}
	-\frac{\partial}{\partial Z}\hat{\mu}\frac{\partial\varphi_1}{\partial Z}-\ii|\xi|\left(\frac{\partial}{\partial Z}(\hat{\mu}\varphi_3)+\hat{\lambda}\frac{\partial}{\partial Z}\varphi_3\right)+(\hat{\lambda}+2\hat{\mu})|\xi|^2\varphi_1=\Lambda\varphi_1 ,
\end{equation}
\begin{equation}\label{Rayleigh2}
	-\frac{\partial}{\partial Z}(\hat{\lambda}+2\hat{\mu})\frac{\partial\varphi_3}{\partial Z}-\ii|\xi|\left(\frac{\partial}{\partial Z}(\hat{\lambda}\varphi_1)+\hat{\mu}\frac{\partial}{\partial Z}\varphi_1\right)+\hat{\mu}|\xi|^2\varphi_3=\Lambda\varphi_3 ,
\end{equation}
with free-surface boundary conditions
	\begin{align}
	\ii|\xi| \varphi_3(0) + \frac{\partial \varphi_1}{\partial Z}(0) = 0 ,\label{Rayleighboundary1}
	\\
	\ii\hat{\lambda} |\xi| \frac{\partial \varphi_1}{\partial Z}(0) + (\hat{\lambda} + 2\hat{\mu}) \frac{\partial \varphi_3}{\partial Z}(0)\label{Rayleighboundary2}
	= 0,
\end{align}
where $Z \in \left( - \infty, 0 \right]$ is the coordinate with direction normal to the boundary, $\hat{\mu}$ and $\hat{\lambda}$ are the density-normalized Lam{\'e} parameters, $\xi$ is the dual of the coordinate vector $(x,y)$ parallel to the boundary, $\omega$ is the frequency and $\varphi_1$ and $\varphi_3$ are the components of the displacement vector on the $x$ and $z$  direction. Equation \eqref{Rayleigh1}--\eqref{Rayleigh2} describes the motion of an infinitesimal element of elastic solid that follows an ellipse on the $xz$-plane below the Earth's surface. Equations \eqref{Rayleigh1}--\eqref{Rayleigh2} are obtained in \cite{4authorsPaper} after decoupling the elastic wave equation for infinitesimal solid and using the semiclassical limit. Hence, the operator \eqref{Rayleigh1}--\eqref{Rayleigh2}  is an operator-valued symbol of a semiclassical pseudo-differential operator after using Fourier transform on the directions along the Earth's surface $x$ and $y$ which get replaced by their dual variable $\xi_1$ and $\xi_2$ components of the wave vector $\xi$.

In this paper, we study the Rayleigh boundary value problem that we obtained from decoupling the Hamiltonian. The boundary conditions are free-surface which means that we do not have forces along the Earth surface. In this case, we do not have a Schr{\"o}dinger-type form of the boundary value problem and we cannot define the Jost function and recover the differential operator from it as in \cite{Sottile, Sottile2}. Instead, we perform a Pekeris-Markushevic transform (see \cite{Markushevich1994}) that leads to a Schr\"odinger-type form with eigenvalues $-\xi^2$ and Robin boundary condition depending on the spectral parameter $\xi$. 

This transformed problem is no longer self-adjoint, so we lack some of the properties we had in the Love case (see \cite{Sottile, Sottile2})). Moreover, the Jost function can no longer  be reconstructed by the resonances because it is not entire in the complex plane. Hence, we need to define a function $F(\xi)$ consisting of the product of the Rayleigh determinants of the four different sheets of a Riemann surface (see Section \ref{Riemann surface Rayleigh}) defined from the quasi-momenta $q_P$ and $q_S$. We obtain new results by proving that this function is entire (see Theorem \ref{F is entire}), of exponential-type (see Theorem \ref{F is of expponential type}) and of Cartwright class, establishing explicit estimates on its  indices $\rho_{\pm}(F)$ (see Theorem \ref{F is in Cartwright class}). As an application of these results, we also obtain new direct results on the number of resonances (see Corollary \ref{Levinson theorem Rayleigh}) and the forbidden domain for the resonances (see Theorem \ref{Forbidden domain theorem Rayleigh}) in Section \ref{Direct result Rayleigh section}. Even though the setting is made in order to prove the inverse resonance problem, this is not achieved in this paper.
One of the biggest challenge in this paper is to be able to define the Riemann surface, and the reflection and conjugation on each sheet in a well-defined way. Once we have achieved this, we are able to obtain symmetry properties of the Jost solutions and use the mappings $w_{\bullet}$ (see \eqref{SP mappings}), with $\bullet=P,S,PS$, to pass from one sheet to another one of the Riemann surface (see \cite{Christiansen2005}). The choice of the Riemann surface is also crucial to obtain the correct estimates of the determinants of the Jost function in Section \ref{Subsection Estimates Jost}, which is consistent with having divergent exponentials in the unphysical sheets, as it is for the Jost function in \cite{Sottile} (see also \cite{Sottile2}) and differently than in \cite{Iantchenko2}.

Pekeris was the first to consider the inverse problem of reconstructing the Lam{\'e} parameters and density, given that the displacement vector is known on the boundary and the frequencies of the waves are given. He obtained uniqueness in case the Lam{\'e} parameters and density are real analytic functions (see \cite{Pekeris}). Markushevich continued the work of Pekeris and obtained uniqueness of the reconstruction of the smooth potential from the knowledge of the Jost solution at the boundary (see \cite{Markushevich1992}).
Later, Beals, Henkin and Novikova obtained uniqueness of the reconstruction of the smooth and exponentially decreasing potential from the knowledge of the Weyl function at the boundary (see \cite{Beals} and similarly \cite{Iantchenko}).
In  \cite{Iantchenko2} the authors claim to obtain direct result on the resonances, however some contradictions appear. For example, $q_P$ and $q_S$ are defined to be even function of the wave vector $\xi$ but this choice of the quasi-momenta does not allow for the existence of resonances as the solution to the problem will always be $L^2$ and always have decaying exponentials. In fact, the authors obtain Jost functions that are always bounded (not of exponential type) in every sheet (see  (92), (93) and (94) in \cite{Iantchenko2}) and hence results on  exponential type and Cartwright class properties of the entire  function $F(\xi)$ have just a trivial meaning.

This paper originaes from a collaboration with the author's former supervisor Alexei Iantchenko, which was then interrupted and led to different procedures and results throughout the paper. In this paper we construct a well-defined four-sheeted Riemann surface for which the quasi-momenta $q_P$ and $q_S$ are single valued, holomorphic and odd functions. This leads to the definition of the physical sheet, that is the sheet of the Riemann surface where the eigenvalues lies ($L^2$ eigenfunctions) and the three unphysical sheets, that are the sheets where the resonances lie, i.e. non $L^2$ solutions of the problem. This is in stark contrast to the procedure in \cite{Iantchenko2}.

\setcounter{equation}{0}

\section{Main equations of the Rayleigh problem}

In this paper, we want to study the Rayleigh boundary problem obtained from the decoupling of the Hamiltonian: \eqref{Rayleigh1}--\eqref{Rayleighboundary2}. Let 
\begin{align}\label{Hamiltonian_Rayleigh}
	H \Phi :=& \ma -\frac{\partial}{\partial Z} \left( \mu \frac{\partial}{\partial Z}  \cdot \right) + (\lambda + 2 \mu)|\xi|^2  & -\ii |\xi| \left[ \frac{\partial}{\partial Z}(\mu \cdot)  + \lambda \left( \frac{\partial}{\partial Z}  \cdot \right)\right] \\ \\
	-\ii |\xi| \left[ \frac{\partial}{\partial Z}(\lambda \cdot)  + \mu( \frac{\partial}{\partial Z}  \cdot) \right] & -\frac{\partial}{\partial Z} \left( (\lambda +2\mu) \frac{\partial}{\partial Z}  \cdot \right) +  \mu |\xi|^2 \am \left(\begin{array}{c}
		\varphi_1\\
		\varphi_3
	\end{array}\right),
\end{align}
where $\Phi := (\varphi_1,	\varphi_3)$, and with free boundary conditions
\begin{align}\label{boundary conditions}
	&a_-(\Phi)= \ii \hat{\lambda} |\xi| \varphi_1(0^-)
	+ (\hat{\lambda} + 2\hat{\mu}) \frac{\partial \varphi_3}{\partial Z}(0^-) = 0 \nonumber
	\\
	&b_-(\Phi)= \ii |\xi|\hat{\mu} \varphi_3(0^-) + \hat{\mu}\frac{\partial \varphi_1}{\partial Z}(0^-) = 0.
\end{align}

\begin{assumption}[Homogeneity]\label{Homogeneity assumption}
	We assume that below a certain depth $Z_I$, the medium is homogeneous, so the Lam{\'e} parameters are constant
	\begin{align}\label{homogeneity in depth}
	&	\hat{\mu} \left(Z\right) = \hat{\mu} \left(Z_I \right):=\hat{\mu}_I  \qquad \text{for  } Z \leq Z_I, \\
		& 	\hat{\lambda} \left(Z\right) = \hat{\lambda} \left(Z_I \right):=\hat{\lambda}_I  \qquad \text{for  } Z \leq Z_I.
	\end{align}
\end{assumption}
Below, we give the definition of the Jost solutions and the unperturbed Jost solutions, namely the Jost solution in the case of the homogeneous medium ($\hat{\mu}(Z)$  and $\hat{\lambda}(Z)$ constant).
\begin{definition}[Unperturbed Jost solution]\label{Unperturbed Jost solution Rayleigh}
	If $\hat{\mu}(Z)=\hat{\mu}_I,$  $\hat{\lambda}(Z)=\hat{\lambda}_I$ are constants, we define the unperturbed Jost solutions 
	\begin{align}\label{Jost solution homogeneous case}
		&f_{P,0}^\pm =\left(\begin{array}{c} |\xi| \\
			\pm q_P\end{array}\right)e^{\pm i Zq_P},\qq q_P:=\sqrt{\frac{\omega^2}{\hat{\lambda}_I+ 2\hat{\mu}_I}-|\xi|^2}, \nonumber \\
		&f_{S,0}^\pm=\left(\begin{array}{c}  \pm q_S \\
			-|\xi|\end{array}\right)e^{\pm i Zq_S},\qq q_S:=\sqrt{\frac{\omega^2}{\hat{\mu}_I}-|\xi|^2}, 
	\end{align}
	for $Z<0$. That is, they are fundamental solutions to $H\Phi=\omega^2\Phi$.
\end{definition}

\begin{definition}[Jost solution]\label{Jost solution Rayleigh}
	We define the Jost solutions $f_P^\pm,$ $f_S^\pm$ for $Z<0$ as the fundamental solutions to $H\Phi=\omega^2\Phi$ satisfying the conditions
	$$f_P^\pm=f_{P,0}^\pm,\quad f_S^\pm=f_{S,0}^\pm \qquad \mbox{for}\,\,Z<Z_I, $$ where $f_{P,0}^\pm$ and $f_{S,0}^\pm$ are the solutions in the homogeneous case as in Definition \ref{Unperturbed Jost solution Rayleigh}.
\end{definition}

\section{Riemann surface and mappings}\label{Riemann surface Rayleigh}

We make an analytic continuation of $|\xi|$ to the whole complex plane $ \mathbb{C}$ and denote it by $\xi$. Let the quasi-momenta $q_P$ and $q_S$ be defined as
\begin{align*}
	q_P:=\ii \sqrt{\xi^2 - \frac{\omega^2}{\hat{\lambda}_I+ 2\hat{\mu}_I}}, \qq q_S:=\ii \sqrt{\xi^2 - \frac{\omega^2}{\hat{\mu}_I}}.
\end{align*}
We introduce the branch cuts for $q_P$ along $\left[-\frac{\omega}{\sqrt{\hat{\lambda}_I+ 2\hat{\mu}_I}}, \frac{\omega}{\sqrt{\hat{\lambda}_I+ 2\hat{\mu}_I}}\right] \cup \ii \mathbb{R}$ and for $q_S$ along $\left[-\frac{\omega}{\sqrt{\hat{\mu}_I}}, \frac{\omega}{\sqrt{\hat{\mu}_I}}\right] \cup \ii \mathbb{R}$ as in Figure \ref{Branch cuts} and we choose the branches of the square root such that $q_P(\xi)$, $q_S(\xi) \in \ii \mathbb{R_+}$ when $\xi > \frac{\omega}{\hat{\mu}_I}$ is real. Let $\bullet=P,S$, then $q_{\bullet}(\xi) \in \mathbb{C}_+$ when $\xi$ belongs to the first sheet of the Riemann surface of $q_{\bullet}(\xi)$ and $q_{\bullet}(\xi) \in \mathbb{C}_-$ when $\xi$ belongs to the second sheet of the Riemann surface of $q_{\bullet}(\xi)$. 
\begin{figure}
	\centering
	\includegraphics[scale=0.85]{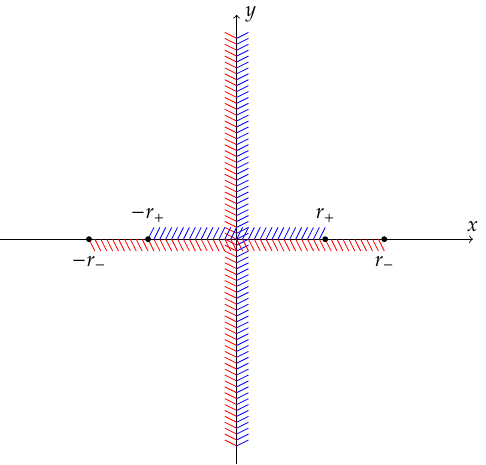}
	\caption{We show the branch cuts in a single sheet of $\Xi$. In the figure $r_-:=\frac{\omega}{\sqrt{\hat{\mu}_I}}$ and $r_+:=\frac{\omega}{\sqrt{\hat{\lambda}_I+ 2\hat{\mu}_I}}$. Blue indicates the branch cuts for the quasi-momentum $q_P$, while red indicates the branch cuts for the quasi-momentum $q_S$.}
	\label{Branch cuts}
\end{figure}
We then consider $q_P$, $q_S$ defined on the joint four-sheeted Riemann surface $\Xi$, defined so that $q_P$ and $q_S$ are single-valued and holomorphic. The  Riemann surface $\Xi$ is obtained by gluing together the following sheets:
\begin{align*}
	\Xi_{\pm, \pm} := \left\lbrace \xi: \pm \im q_P(\xi) >0, \pm \im q_S(\xi) >0 \right\rbrace.
\end{align*}
We denote by $-\xi$ the point in $\Xi$ belonging to the same sheet as $\xi$ obtained by reflecting $\xi$ with respect to the origin, as in Figure \ref{Reflection}. Without loss of generality, we start from a point $\xi$ on the sheet $\Xi_{+,+}$ such that $q_{\bullet}(\xi) \in \mathbb{R_+}$  for $\bullet= P,S$. We approach the imaginary line avoiding the points $r_-:=\frac{\omega}{\sqrt{\hat{\mu}_I}}$ and $r_+:=\frac{\omega}{\sqrt{\hat{\lambda}_I+ 2\hat{\mu}_I}}$, where the quasi-momenta $q_P$ and $q_S$ are not holomorphic. At points $x-\ii0$, $x\in \left(-\frac{\omega}{\sqrt{\hat{\lambda}_I+ 2\hat{\mu}_I}}, \frac{\omega}{\sqrt{\hat{\lambda}_I+ 2\hat{\mu}_I}}\right)$ we have $q_{\bullet}(x-\ii0) \in \mathbb{R_+}$. When we pass through the imaginary line we end up on the sheet $\Xi_{-,-}$ (dashed line in Figure\ref{Reflection}) and then we return to the sheet $\Xi_{+,+}$ after passing through the cut $\left[-\frac{\omega}{\sqrt{\hat{\lambda}_I+ 2\hat{\mu}_I}}, \frac{\omega}{\sqrt{\hat{\lambda}_I+ 2\hat{\mu}_I}}\right]$. Finally, we reach the point $-\xi$ avoiding again $-r_-$ and $-r_+$, as shown in Figure \ref{Reflection}.
\begin{figure}
	\centering
	\includegraphics[scale=1.2]{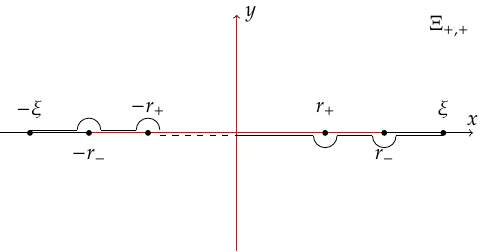}
	\caption{Reflection from $\xi$ to $-\xi$ in the physical sheet $\Xi_{+,+}$. The red lines represent the cuts of the Riemann sheets.}
	\label{Reflection}
\end{figure}

Since the rotations around $r_{\pm}$ and $-r_{\mp}$ are in opposite directions we see that $q_{\bullet}(-\xi) = q_{\bullet}(\xi)$, $\bullet= P,S$.
Hence, it suffices to study $q_P$ and $q_S$ for $\xi \in \Xi$ such that $\re \xi \geq 0$.  For large $\xi$ with $\re \xi \geq0$ we can write
\begin{align}\label{Quasi momenta in the 4 sheets}
	&q_P(\pm \xi) = \ii \xi + \mathcal{O}(|\xi|^{-1}), \qq \qq &\xi \in \Xi_{+,\pm}, \re \xi \geq 0, \nonumber \\
	&q_P(\pm \xi) = -\ii \xi + \mathcal{O}(|\xi|^{-1}), \qq \qq &\xi \in \Xi_{-,\pm}, \re \xi \geq 0, \nonumber \\
	&q_S(\pm \xi) = \ii \xi + \mathcal{O}(|\xi|^{-1}), \qq \qq &\xi \in \Xi_{\pm ,+}, \re \xi \geq 0, \nonumber \\
	&q_S(\pm \xi) = -\ii \xi + \mathcal{O}(|\xi|^{-1}), \qq \qq &\xi \in \Xi_{\pm,-}, \re \xi \geq 0.
\end{align}
For example if $\xi \in \Xi_{+,+}$, but $\re \xi <0$, then we have $q_P(\pm \xi) = -\ii \xi + \mathcal{O}(|\xi|^{-1}) $ instead, and similarly for $q_S$.

We can define the mappings $w_P,w_S,w_{PS}: \Xi \to \Xi$, where each mapping applied to $\xi$ can change the sign of either one or both the imaginary parts of the quasi-momenta, hence it maps points from one to another fold of the Riemann surface (see \cite{Christiansen2005}). These mappings operate according to the rules 
\begin{subequations}\label{SP mappings}
\begin{align}
	&q_S(w_S(\xi))= -q_S(\xi),\qq &q_P(w_S(\xi))=q_P(\xi); \\
	& q_S(w_P(\xi))= q_S(\xi),\qq &q_P(w_P(\xi))=- q_P(\xi); \\
	&q_S(w_{PS}(\xi))= -q_{S}(\xi),\qq &q_P(w_{PS}(\xi))=-q_P(\xi).
\end{align}
\end{subequations}

For any value of $\xi$ it holds that
\begin{equation}
	q_{\bullet}(\xi) =  - \overline{q_{\bullet}(\overline{\xi})}, 
\end{equation}
where the conjugation of $\xi$ is defined as a normal conjugation in a single sheet of $\Xi$ by contours not passing through the cuts.

\begin{lemma}\label{Lemma sheets of Riemann surface}
	We have 
	\begin{enumerate}
		\item $\im \left(q_P + q_S\right) >0$ and $\im \left(q_P - q_S\right)>0$ in $\Xi_{+, \pm}$,
		\item $\im \left(q_P + q_S\right) <0$ and $\im \left(q_P - q_S\right)<0$ in $\Xi_{-, \pm}$.
	\end{enumerate}
\end{lemma}
\begin{proof}
	We know that
	\begin{align*}
		&\im (q_P + q_S) = \im \left[ \frac{(q_P + q_S) (q_P - q_S)}{q_P - q_S} \right] = (q_P^2 - q_S^2) \im \left( \frac{1}{q_P - q_S}\right) = - \frac{\omega^2 (\hat{\lambda} + \hat{\mu})}{\hat{\mu} (\hat{\lambda} + 2 \hat{\mu})} \im \left( \frac{1}{q_P - q_S}\right).
	\end{align*}
	Since $\frac{\omega^2 (\hat{\lambda} + \hat{\mu})}{\hat{\mu} (\hat{\lambda} + 2 \hat{\mu})}>0$ and $\sgn \left(\im \left( 1 / z\right) \right) = - \sgn \left(\im (z)\right)$, then $\sgn \left( \im (q_P + q_S) \right) = \sgn \left(\im (q_P - q_S) \right)$. Since $\im \left(q_P + q_S\right)>0$ in $\Xi_{+,+}$ and $\im \left(q_P + q_S\right)<0$ in $\Xi_{-,-}$, while $\im \left(q_P - q_S\right)>0$ in $\Xi_{+,-}$ and $\im \left(q_P - q_S\right)<0$ in $\Xi_{-,+}$, the lemma follows.
\end{proof}

\section{Cartwright class functions}

In this subsection we give some definitions and results from complex analysis that will be useful later on (see \cite[Chapter 3]{Koosis1988} and \cite[Chapter 1]{Levin}).
\begin{definition}[Exponential type function]
	An entire function $f(z)$ is said to be of exponential type if there are real constants $\alpha$, $o$ and $A$ such that 
	\begin{align}\label{Exponential type}
		|f(z)| \leq A e^{\alpha |z|^o}
	\end{align}
	for $z\to \infty$ in the complex plane. The infimum of the $o$  and $\alpha$ such that \eqref{Exponential type} is satisfied are called respectively \textit{order} and \textit{type} of the exponential type. 
\end{definition}

\begin{definition}[Cartwright class]\label{Cartwright class definition}
	A function $f$ is said to be in the Cartwright class with indices $\rho_{+}=A$ and $\rho_-=B$, if $f(z)$ is entire, of exponential type, and the following conditions are fulfilled:
	\begin{equation}\label{Definition Cartwright class}
		\int_{\mathbb{R}} \frac{\log^+ |f(x)| dx}{1 + x^2} < \infty, \quad \rho_+(f) = A, \quad \rho_-(f) = B 
	\end{equation} 
	where  $\rho_{\pm}(f) \equiv \lim \sup_{y\to \infty} \frac{\log|f(\pm iy)|}{y}$ and  $\log^+(x) =\max \left\lbrace \log x,0 \right\rbrace$.
\end{definition}
Basically, for a function to be of Cartwright class means that it is of exponential order 1, of type $A$ in the upper half-plane and $B$ in the lower half-plane and with positive part of the logarithm of its absolute value integrable with respect to the Poisson measure $d\Pi(t) = \frac{dt}{1 + t^2}$.
 A useful application of the Cartwright class property is the Levinson theorem (see \cite[page 69]{Koosis1988}), which is the counterpart of the Weyl law for the resonances. We denote by $\mathcal{N}_+(r,f)$ the number of zeros of an entire function $f$ with positive imaginary part with modulus $\leq r$, and by $\mathcal{N}_-(r,f)$ the number of zeros with negative imaginary part having modulus $\leq r$.  Moreover, $\mathcal{N}_{\pm}(f) :=\lim_{r \to \infty} \mathcal{N}_{\pm}(r,f)$. The total number of zeros with modulus smaller than $r$ is $\mathcal{N}(r,f) := \mathcal{N}_+(r,f) + \mathcal{N}_-(r,f)$.
\begin{theorem}[Levinson]\label{Levinson}
	Let the function $f$ be in the Cartwright class with $\rho_{+}=\rho_{-}=A$ for some $A>0$. Then 
	\begin{equation*}
		\mathcal{N}_{\pm}(r,f) = \frac{Ar}{\pi}\left(1 + o(1)\right) \qq \text{for } r\to \infty.
	\end{equation*}
	Given $\delta>0$, the number of zeros of $f$ with modulus $\leq r$ lying outside both of the two sectors $|\arg z |<\delta$, $|\arg z -\pi|< \delta$ is $o(r)$ for large $r$.
\end{theorem}

\section{Parity properties}\label{parity properties section}

We define the differential operators $a=a (Z,D_Z, \xi)$ and $b=b(Z,D_Z, \xi)$ such that if we apply them to the function $\Phi$, which belongs in the space of the solutions of $(H-\omega^2)\Phi=0$, it gives
\begin{align*}
	&a(\Phi) (Z,\xi) := \ii \xi  \hat{\lambda}(Z) \varphi_1(Z)
	+ \left( \hat{\lambda}(Z) + 2\hat{\mu}(Z) \right) \frac{\partial \varphi_3}{\partial Z}(Z),\\
	&b(\Phi) (Z,\xi) = \ii \xi\hat{\mu}(Z) \varphi_3(Z) + \hat{\mu}(Z)\frac{\partial \varphi_1}{\partial Z}(Z)
\end{align*}
so $a_-(\Phi) = a(\Phi) (0^-,\xi)$ and $b_-(\Phi)
= b(\Phi) (0^-,\xi)$ with boundary conditions $a(\Phi)$ and $b(\Phi)$ given by \eqref{boundary conditions}.  In the following lemma we show the symmetries of the Jost solution. The idea of looking for symmetry was inspired by \cite{CohenKappeler1985}, where symmetries of the reflection coefficients are found instead.
\begin{lemma}\label{Lemma symmetry Jost solutions}
	On the Riemann surface $\Xi$, using the projection mappings to the sheets of $\Xi$ we get:
	\begin{align}\label{symmetry on Riemann surface}
		&f_P^\pm(Z,w_{P}(\xi))=f_P^\pm(Z,w_{PS}(\xi))=f_P^\mp(Z,\xi), \nonumber\\
		&f_S^\pm(Z,w_{S}(\xi))=f_S^\pm(Z,w_{PS}(\xi))=f_S^\mp(Z,\xi),\nonumber \\
		&f_P^\pm(Z,w_{S}(\xi)) = f_P^\pm(Z,\xi) \nonumber \\
		&f_S^\pm(Z,w_{P}(\xi)) = f_S^\pm(Z,\xi).
	\end{align}
\end{lemma}
\begin{proof}
	
	We can prove \eqref{symmetry on Riemann surface} using that $H(Z, \xi)$ is invariant under the projection mappings $w_S$, $w_P$ and $w_{P,S}$ as it is independent on the quasi-momenta $q_S$ and $q_P$, while the unperturbed solutions $f_{\bullet,0}$, $\bullet=P,S$, satisfy those properties.
\end{proof}
Looking at the boundary conditions \eqref{boundary conditions}, we define
\begin{equation}\label{boundary matrix}
	B(Z,\xi) := B(Z, D_{Z}, \xi) = \begin{pmatrix}
		\ii \hat{\lambda}\xi   &\left(\hat{\lambda} + 2\hat{\mu}\right)\frac{\partial}{\partial Z}\\
		\hat{\mu}\frac{\partial}{\partial Z}   & \ii \hat{\mu}\xi 
	\end{pmatrix}.
\end{equation}
Then \eqref{boundary conditions} is equivalent to
\begin{equation*}
	\left. B(Z,\xi) \left(\begin{array}{c}
		\varphi_1(Z,\xi)\\
		\varphi_3(Z,\xi)
	\end{array}\right)\right\rvert_{Z=0} = 0.
\end{equation*}

\begin{lemma}\label{Parity properties of the boundary conditions applied on Jost sol}
	Using the projection mappings to the sheets of $\Xi$ we get
	\begin{align}\label{symmetry boundary cond on Riemann surface}
		&a(f_S^\pm)(w_S(\xi)) = a(f_S^\pm)(w_{PS}(\xi)) = a(f_S^\mp)(\xi), \nonumber \\
		&b(f_S^\pm)(w_S(\xi)) = b(f_S^\pm)(w_{PS}(\xi)) = b(f_S^\mp)(\xi), \nonumber \\
		&a(f_P^\pm)(w_P(\xi)) = a(f_P^\pm)(w_{PS}(\xi)) = a(f_P^\mp)(\xi), \nonumber \\
		&b(f_P^\pm)(w_P(\xi)) = b(f_P^\pm)(w_{PS}(\xi)) = b(f_P^\mp)(\xi).
	\end{align}
\end{lemma}
\begin{proof}
 The proof of \eqref{symmetry boundary cond on Riemann surface} is straightforward from \eqref{symmetry on Riemann surface} as the coefficients in the boundary conditions do not depend on the quasi-momenta $q_P$ and $q_S$.
		\qedhere
\end{proof}

	\section{Representation of the Jost solution}\label{Section repr of Jost sol}

In the homogeneous case $\hat{\lambda}(Z)=\hat{\lambda}_I$, $\hat{\mu}(Z)=\hat{\mu}_I$, we define
\begin{align*}
	&\theta_{P,0} := \frac{1}{2} (f_{P,0}^+ + f_{P,0}^-), \qquad \varphi_{P,0} := \frac{1}{2q_P} (f_{P,0}^+ - f_{P,0}^-), \\
	& \theta_{S,0} := \frac{1}{2} (f_{S,0}^+ + f_{S,0}^-), \qquad \varphi_{S,0} := \frac{1}{2q_S} (f_{S,0}^+ - f_{S,0}^-),
\end{align*}
and in particular we get
\begin{equation}\label{Definition of theta 0}
	\theta_{P,0} = \ma \xi \cos (q_P Z) \\ \ii q_P \sin (q_P Z) \am, \qquad \varphi_{P,0} = \ma  \frac{\ii \xi }{q_P } \sin (q_P Z) \\  \cos (q_P Z) \am,
\end{equation}
and 
\begin{equation}\label{Definition of varphi 0}
	\theta_{S,0} = \ma i q_S \sin (q_S Z)  \\ -\xi \cos (q_S Z) \am, \qquad \varphi_{S,0} = \ma \cos (q_S Z)  \\  - \frac{i \xi }{q_S } \sin (q_S Z) \am.
\end{equation}
By expanding $\sin (q_{S,P} Z)$ and $\cos (q_{S,P} Z)$ in Taylor series, we see that the functions $\theta_{P,0}^{\pm}$, $\theta_{S,0}^{\pm}$, $\varphi_{P,0}^{\pm}$, and $\varphi_{S,0}^{\pm}$ are entire in $\xi $.
Indeed, after the series expansion only even powers of $q_P$ and $q_S$ appear, hence only natural powers of $\xi$. 
We can express the Jost solution $f_{P,0}^{\pm}$ and $f_{S,0}^{\pm}$ in terms of the functions $\theta_{P,0}^{\pm}$, $\theta_{S,0}^{\pm}$, $\varphi_{P,0}^{\pm}$, and $\varphi_{S,0}^{\pm}$:
\begin{equation}\label{Definition of Jost solution in terms of theta and varphi}
	f_{P,0}^{\pm} = \theta_{P,0} \pm q_P \varphi_{P,0}, \qquad f_{S,0}^{\pm} = \theta_{S,0} \pm q_S \varphi_{S,0}.
\end{equation}
Similarly we can define in the inhomogeneous case

\begin{align}\label{Definition of auxiliary function}
	&\theta_{P} := \frac{1}{2} (f_{P}^+ + f_{P}^-), \qquad \varphi_{P} := \frac{1}{2q_P} (f_{P}^+ - f_{P}^-), \\
	& \theta_{S} := \frac{1}{2} (f_{S}^+ + f_{S}^-), \qquad \varphi_{S} := \frac{1}{2q_S} (f_{S}^+ - f_{S}^-). \label{Definition of theta and varphi}
\end{align}
Note that also here
\begin{equation*}
	f_{P}^{\pm} = \theta_{P} \pm q_P \varphi_{P}, \qquad f_{S}^{\pm} = \theta_{S} \pm q_S \varphi_{S},
\end{equation*}
and the functions $\theta_{P,S}$ and $\varphi_{P,S}$ satisfy the  conditions
\begin{align}\label{Boundary cond var and theta}
	&\theta_{P}= \theta_{P,0},  \qquad \varphi_{P}= \varphi_{P,0},  \qquad Z \leq Z_I \nonumber \\
	&\theta_{S}= \theta_{S,0},  \qquad \varphi_{S}= \varphi_{S,0},  \qquad Z \leq Z_I. 
\end{align}

In particular, the Jost solutions $f_{P}^{\pm}$ and $f_{S}^{\pm}$ are uniquely determined by those auxiliary functions $\theta_{P,0}$, $\theta_{S,0}$, $\varphi_{P,0}$, $\varphi_{P,0}$. For those auxiliary functions we have the following results.

\begin{lemma}\label{Theta and varphi entire}
	The function $\theta_{P}$, $\theta_{S}$, $\varphi_{P}$ and $\varphi_{S}$ are entire and invariant under the mappings $w_{P}(\xi)$, $w_{S}(\xi)$, $w_{PS}(\xi)$:
	\begin{align*}
		&\theta_{P} (\xi) = \theta_{P}(w_{P}(\xi)) = \theta_{P}(w_{S}(\xi)) = \theta_{P}(w_{PS}(\xi)), \\  
		&\theta_{S} (\xi)= \theta_{S}(w_{P}(\xi)) = \theta_{S}(w_{S}(\xi)) = \theta_{S}(w_{PS}(\xi)), \\  
		&\varphi_{P} (\xi) = \varphi_{P}(w_{P}(\xi)) = \varphi_{P}(w_{S}(\xi)) = \varphi_{P}(w_{PS}(\xi)), \\
		&\varphi_{S} (\xi) = \varphi_{S}(w_{P}(\xi)) = \varphi_{S}(w_{S}(\xi)) = \varphi_{S}(w_{PS}(\xi)).
	\end{align*}
\end{lemma}
\begin{proof}
	The functions $\theta_{P}$, $\theta_{S}$, $\varphi_{P}$ and $\varphi_{S}$ are the unique solutions of an ordinary differential equation \eqref{Rayleigh1}--\eqref{Rayleigh2} with analytic dependence on the parameter $\xi$ and satisfying the boundary conditions \eqref{Boundary cond var and theta}, where $\theta_{P,0}$, $\theta_{S,0}$, $\varphi_{P,0}$ and $\varphi_{S,0}$ are entire in $\xi$. Then $\theta_{P}$, $\theta_{S}$, $\varphi_{P}$ and $\varphi_{S}$ are also entire (see \cite[Theorem 2.5.1]{Zettl} or \cite[Theorem 8.4, Chapter 1]{Coddington}).
	
	The invariance under the mapping $w_{\bullet}$, with $\bullet=P,S, PS$, is a consequence of the fact that we can express those auxiliary functions as even powers of the quasi-momenta $q_P$ and $q_S$, knowing that the mapping $w_P$, $w_S$ and $w_{PS}$ only change the sign of the quasi-momenta.
\end{proof}

Since the auxiliary functions are written in terms of the Jost solutions (see \eqref{Boundary cond var and theta}), we can translate Lemma \ref{Parity properties of the boundary conditions applied on Jost sol} in terms of the auxiliary functions as in the following lemma.

\begin{lemma}\label{Parity of the boundary condition}
	The functions $a(\theta_{P}), b(\theta_{P}), a(\theta_{S}), b(\theta_{S})$, $a(\varphi_{P}), b(\varphi_{P}), a(\varphi_{S}), b(\varphi_{S})$ are entire  in $\xi$.
\end{lemma}
\begin{proof}
	By Lemma \ref{Theta and varphi entire} and since the boundary conditions  \eqref{Rayleighboundary1}--\eqref{Rayleighboundary2} have analytic dependence on the parameter $\xi$, then  $a(\theta_{\bullet}), b(\theta_{\bullet}), a(\varphi_{\bullet}), b(\varphi_{\bullet})$, with $\bullet=P,S$, are also entire.

\end{proof}

%

The matrix $B(Z,\xi) $ applied to a matrix, whose columns are respectively $f_P^-$ and $f_S^-$, is
\begin{equation*}
	\left.B(Z,\xi) \ma (f_P^-)_1 & (f_S^-)_1 \\ (f_P^-)_2 & (f_S^-)_2 \am \right\rvert_{Z=0}= \ma a(f_P^-) & a(f_S^-) \\ b(f_P^-) & b(f_S^-) \am,
\end{equation*}
this motivates the definition of the boundary matrix $B(\xi)$ as it follows.
\begin{definition}[Boundary matrix]\label{Definition boundary matrix}
	We define the \textit{boundary matrix} $B(\xi)$ as the quantity
	\begin{equation}\label{Jost matrix function}
		B(\xi) := \ma a(f_P^-) &a(f_S^-)\\  \\
		b(f_P^-) &b(f_S^-) \am.
	\end{equation}
\end{definition}
The eigenvalues and resonances correspond to the zeros of the determinant of the boundary matrix, which we call the Rayleigh determinant.

\begin{definition}\label{Rayleigh determinant def}
	We define the \textit{Rayleigh determinant} as the determinant of the boundary matrix
	\begin{align*}
		\Delta(\xi) := \det \ma a(f_P^-) &a(f_S^-)\\  \\
		b(f_P^-) &b(f_S^-) \am = a(f_P^-) b(f_S^-) - a(f_S^-) b(f_P^-).
	\end{align*}
\end{definition}

As in \cite{Sottile, Sottile2} for the Love problem, we can distinguish between eigenvalues and resonances whether they are zeros of the Rayleigh determinant located in the physical or unphysical sheet. The physical sheet and the unphysical sheets are determined by imaginary part of the quasi-momenta $q_S$ and $q_P$ being positive or negative and leading to $L^2$ or not $L^2$ solutions respectively.

\begin{definition}\label{Definition eigenvalues and reson Rayleigh}
	We introduce the following discrete sets on the Riemann surface $\Xi$ 
	\begin{align*}
		&\Lambda_{+,+}:=\left\{\xi \in \Xi;\qq \Delta(\xi)=0,\qq \im q_P(\xi)>0,\qq \im q_S(\xi)>0\right\},\\
		&\Lambda_{+,-}:=\left\{\xi\in \Xi;\qq \Delta(\xi)=0,\qq \im q_P(\xi)>0,\qq \im q_S(\xi)<0\right\},\\
		&\Lambda_{-,+}:=\left\{\xi \in \Xi;\qq \Delta(\xi)=0,\qq \im q_P(\xi)<0,\qq \im q_S(\xi)>0\right\},\\
		&\Lambda_{-,-}:=\left\{\xi \in \Xi;\qq \Delta(\xi)=0,\qq \im q_P(\xi)<0,\qq \im q_S(\xi)<0\right\},
	\end{align*}
	and the union of them denoted as $\Lambda$
	\begin{align*}
		\Lambda := \bigcup_{a,b=\pm} \Lambda_{a,b}.
	\end{align*}
\end{definition}
It follows from Definition \ref{Definition eigenvalues and reson Rayleigh} that the eigenvalues correspond to the set $\Lambda_{+,+}$, zeros of $\Delta(\xi)$ in the physical sheet, and the resonances correspond to the union of the three remaining sets, which are the zeros of $\Delta(\xi)$ in the unphysical sheet.

The Rayleigh determinant can be written in terms of the auxiliary functions $\theta$ and $\varphi$ in the following way:
\begin{align}
	\Delta (\xi)&= \left[ a(\theta_{P}) - q_P a(\varphi_{P}) \right] \left[ b(\theta_{S}) - q_S b(\varphi_{S}) \right] -  \left[ a(\theta_{S}) - q_S a(\varphi_{S}) \right]  \left[ b(\theta_{P}) - q_P b(\varphi_{P}) \right]  \nonumber \\
	& = d_1 + q_P d_2 + q_S d_3 + q_S q_P d_4 \label{Definitiion of delta in terms of di}
\end{align}
where the coefficients $d_1, d_2, d_3, d_4$ are entire in $\xi$ and are defined as
\begin{align}\label{definition of d1 d2 d3 d4}
	 d_1 &:= a(\theta_{P})b(\theta_{S}) - a(\theta_{S}) b(\theta_{P}), &  d_2 &:= - \left[ a(\varphi_{P})b(\theta_{S}) - a(\theta_{S}) b(\varphi_{P}) \right], \nonumber \\
	d_3 &:= - \left[ a(\theta_{P})b(\varphi_{S}) - a(\varphi_{S}) b(\theta_{P}) \right], &
	 d_4 &:=  a(\varphi_{P})b(\varphi_{S}) - a(\varphi_{S}) b(\varphi_{P}).
\end{align}

The definition of the functions $d_{i}(\xi)$, for $i=1,\ldots,4$, is made purely to describe $\Delta(\xi)$ in terms of entire functions in a simpler form. The formula \eqref{Definitiion of delta in terms of di} shows how $\Delta (\xi)$ depends on the quasi-momenta $q_P$ and $q_S$ and makes it easier to apply the mappings $w_{\bullet}$, with $\bullet=P,S, PS$.  Indeed, we get 
\begin{align}\label{Mapping of the Rayleigh determinant}
	&\Delta (\xi) =  d_1 + q_P d_2 + q_S d_3 + q_S q_P d_4, &
	&\Delta (w_P(\xi))= d_1 - q_P d_2 + q_S d_3 - q_S q_P d_4, \nonumber \\
	&\Delta (w_S(\xi))= d_1 + q_P d_2 - q_S d_3 - q_S q_P d_4, &
	&\Delta (w_{PS}(\xi))= d_1 - q_P d_2 - q_S d_3 + q_S q_P d_4.
\end{align}

\subsection{The entire function \texorpdfstring{$F(\xi)$}{F}}\label{Section F entire}

We want to consider a function $F(\xi)$ that is entire on the complex plane and whose zeros correspond to the eigenvalues and the resonances of the operator associated to \eqref{Hamiltonian_Rayleigh}. From estimates of this function we can obtain estimates of the resonances, which tells us in which areas they are localized. 
\begin{theorem}\label{F is entire}
	The function
	\begin{equation}\label{equation entire funct F}
		F(\xi) = \Delta(\xi) \Delta(w_P(\xi))  \Delta(w_S(\xi))  \Delta(w_{PS}(\xi)),
	\end{equation}
	is entire.
\end{theorem}
\begin{proof}
	Using the definitions of the Rayleigh determinants as in \eqref{Mapping of the Rayleigh determinant} we can evaluate:
	\begin{align*}
		F(\xi) &= \left( d_1 + q_P d_2 + q_S d_3 + q_S q_P d_4 \right) \left( d_1 - q_P d_2 - q_S d_3 + q_S q_P d_4 \right) \\
		&\phantom{=\;}\cdot \left( d_1 - q_P d_2 + q_S d_3 - q_S q_P d_4 \right) \left(  d_1 + q_P d_2 - q_S d_3 - q_S q_P d_4 \right) \\
		 &= d_1^4+q_P^4d_2^4+q_S^4d_3^4+q_S^4q_P^4d_4^4-2q_P^2d_1^2d_2^2-2q_S^2d_1^2d_3^2 -2q_S^2q_P^2d_1^2d_4^2 \nonumber\\ &\phantom{=\;}-2q_P^2q_S^2d_2^2d_3^2-2q_P^4q_S^2d_2^2d_4^2-2q_S^4q_P^2d_3^2d_4^2+8q_P^2q_S^2d_1d_2d_3d_4.
	\end{align*}
	Since $d_1$, $d_2$, $d_3$, $d_4$ are entire functions in $\xi$ and since $\xi$ is present inside the quasi-momenta, which always have even power, $F(\xi)$ is an entire function for $\xi \in \mathbb{C}$, whose zeros correspond to the set $\Lambda$.
\end{proof}

	\section{The Pekeris-Markushevich transform}\label{Pekeris-Markushevich transform}

In this section we recall some known facts of \cite{Markushevich1994} and we use the same notation as in \cite{Iantchenko, Argatov}. In particular, we consider from now on $x:=-Z\in \left[0,+\infty\right[$ as the independent variable, and we let $H:=-Z_I$. We make some substitutions in the differential equations such that we can pass from a self-adjoint differential operator to a not self-adjoint Sturm-Liouville problem where the spectral parameter $\xi$ is also present in the boundary condition and the frequency $\omega$ is present in the potential and in the boundary condition. 

Basically, we lose self-adjointness of the problem but we gain a Schr{\"o}dinger-type differential equation which can help us in computing the Jost solution through a Volterra-type integral equation. The adjoint problem has transposed potential and boundary condition. The Rayleigh system \eqref{Rayleigh1}--\eqref{Rayleigh2} with boundary condition \eqref{Rayleighboundary1}--\eqref{Rayleighboundary2} gets transformed into the Sturm-Liouville boundary value problem
	\begin{equation*}
			-F''+Q(x) F=-\xi^2 F,
\end{equation*}
\begin{equation*}
	F'+\Theta F=0,\quad x=0,
\end{equation*}
	where $Q(x)$ is the potential and $\Theta$ is defined as
	\begin{align}\label{1rw(2.15)new}
		\Theta(\xi) &=\left(D^a(\xi) \right)^{-1} C^a(\xi) =\left(\begin{array}{cc}
			\displaystyle
			-\frac{\hat{\mu}'(0)}{\hat{\mu}(0)} &  \displaystyle
			\frac{1}{2\hat{\mu}_I}\frac{\hat{\mu}^2(0)}{(\hat{\lambda}(0)+2\hat{\mu}(0))}
			\\ \displaystyle
			\frac{\hat{\mu}_I}{\hat{\mu}(0)}\biggl(2\xi^2-\frac{\omega^2}{\hat{\mu}(0)}
			-\hat{\mu}(0) \frac{d^2}{dx^2}\biggl(\frac{1}{\hat{\mu}}
			\biggr) (0)\biggr)
			& 0
		\end{array}\right)  \nonumber \\
		&=:\ma-\theta_3 &\theta_2\\ \\
		2\frac{\hat{\mu}_I}{\hat{\mu}}\xi^2-\theta_1&0\am.
	\end{align}
We denote the restriction of $Q(x)$ to the region  $x\geq H $ by $Q_0(x)$. Then we can write
\begin{align}\label{1rw(3.11)new}
	Q_0(x)&=\omega^2\ma-\frac{1}{\hat{\mu}_I} &0\\ 0 &-\frac{1}{\hat{\lambda}_I+2\hat{\mu}_I}\am+\omega^2\frac{\hat{\lambda}_I+\hat{\mu}_I}{\hat{\mu}_I(\hat{\lambda}_I+2\hat{\mu}_I)}\ma -G_{12}^HG_{21}(x)&G_{21}(x)G_{11}^H\\ -G_{12}^HG_{22}(x) &G_{12}^HG_{21}(x)\am 
\end{align}
$Q_0(x)$ extended to $\left[ 0, \infty \right)$ is called background potential. In \eqref{1rw(3.11)new} $G_{ij}^{H}$ satisfy $G_{11}^{H}G_{22}^{H} - G_{12}^{H}G_{21}^{H} = 1$ and $G_{ij}(x)$ are defined as
\begin{align}	\label{1rw(3.9)new}
	G_{11}(x)&:=G_{11}^H,  & G_{12}(x)&:=G_{12}^H, \nonumber\\
	G_{21}(x)&:=-\frac{c_I}{2}G_{11}^H(x-H)+G_{21}^H, & G_{22}(x)&:=-\frac{c_I}{2}G_{12}^H(x-H)+G_{22}^H.
\end{align}
We introduce the perturbed potential $V(x):=Q(x)-Q_0(x)$ which satisfies $V(x)=0$ for $x\geq H$ and we can write the transformed Rayleigh system as:
 	\begin{align}
 	-F''+Q_0(x) F + V(x) F &=-\xi^2 F, \label{StLgen_bis} \\
 		F'+\Theta F&=0, \qquad \qquad \qquad x=0.	\label{1rw(2.12)bis}
 \end{align}
We reduced the original Rayleigh system to two matrix Sturm-Liouville problems with mutually transposed potentials and boundary conditions (see \cite{Pekeris}), as also shown by Markushevich in \cite{Markushevich1987,Markushevich1994,Markushevich1992}.

					\subsection{Jost solutions and Jost function }

					In this subsection, we define the Jost solution and the Jost function following the notation of \cite{Iantchenko}.
					We construct solutions to the equation
					\[\label{StLgen_0}-F''+Q_0F=-\xi^2F\]
					of the form
					\begin{align*}F^\pm_{P,0}=&\ma (F^\pm_{P,0})_1 \\ (F^\pm_{P,0})_2 \am e^{\pm i xq_P},\qq q_P=\sqrt{\frac{\omega^2}{\hat{\lambda}_I+2\hat{\mu}_I}-\xi^2},\\
						F^\pm_{S,0}=&\ma (F^\pm_{S,0})_1\\  (F^\pm_{S,0})_2\am e^{\pm i xq_S},\qq q_S=\sqrt{\frac{\omega^2}{\hat{\mu}_I}-\xi^2}. \end{align*}

							The Jost solutions of (\ref{StLgen_bis}) are given by the conditions
							$$F^\pm_{P}=F^\pm_{P,0}.\qq F^\pm_{S}=F^\pm_{S,0}\qq\mbox{for}\qq x>H. $$ 
							
							We define the matrix Jost solution
							$$\mathcal{F}(x,\xi)=\ma \left( F^+_P\right)_1  & \left( F^+_S \right)_1 \\ \\ \left( F^+_P\right)_2 & \left( F^+_S\right)_2 \am  ,$$ 
							and the unperturbed matrix Jost solution
							\begin{align*}
								\mathcal{F}_0(x,\xi) &=\ma \left( F^+_{P,0}\right)_1  & \left( F^+_{S,0} \right)_1 \\ \\ \left( F^+_{P,0}\right)_2 & \left( F^+_{S,0}\right)_2 \am =
								\left(\begin{array}{cc}
									\left(G_{21}(x) + iq_P\frac{\hat{\mu}_I}{\omega^2}G_{11}^H\right) e^{-\ii q_Px}  & -\frac{\hat{\mu}_I \xi}{\omega^2}
									G_{11}^H e^{-\ii q_Sx} 
									\\ \\
									\left(G_{22}(x) + iq_P\frac{\hat{\mu}_I}{\omega^2}G_{12}^H\right) e^{-\ii q_Px}  & -\frac{\hat{\mu}_I \xi}{\omega^2}
									G_{12}^H e^{-\ii q_Sx} 
								\end{array}\right)
							\end{align*}
							
							where $\left( F^+_{P}\right)_i$ denotes the $i$ component of the vector $F^+_{P}$ and similarly  for $F^+_{S}$. The matrix Jost solution satisfies the Volterra-type integral equation
							$$\mathcal{F}(x, \xi)=\mathcal{F}_{0}(x,\xi)-\int_x^\infty \mathcal{G}(x,y) V(y)\mathcal{F}(y,\xi)dy,$$ 
							$\mathcal{G}(x,y)$ is the Green function, which is obtained such that each column of $\mathcal{G}(\cdot,t)$   satisfies the unperturbed equation \[\label{AStLunpert}-F''+Q_0F=-\xi^2F,\] and the conditions 
							\begin{equation}\label{Conditions for Green kernel}
								\mathcal{G}(x,x)=0,\qq\frac{\partial}{\partial x}\mathcal{G}(x,y)_{|y=x}=I.
							\end{equation}
							Explicitly the Green function is given by (see \cite{Iantchenko})
							\begin{align}\label{Definition of the Green kernel}
								&\mathcal{G}(x,y) =\mathcal{A}(x)\frac{\sin((x-y)q_P)}{q_P}+\mathcal{B}(y)\frac{\sin((x-y)q_S)}{q_S} +\mathcal{C} \frac{\cos( (x-y)q_S)-\cos((x-y) q_P)}{\omega^2} 
							\end{align}
							where 
							\begin{align}\label{Definition of A B and C}
								\mathcal{A}(x) :&= \ma -G_{12}^H G_{21}(x) &
								G_{11}^H G_{21}(x)
								\\ \\ - G_{12}^H G_{22}(x)  &G_{11}^H G_{22}(x)
								\am, &
								\mathcal{B}(y) :&= \ma
								G_{11}^H G_{22}(y) & -G_{11}^H G_{21}(y)
								\\ \\
								G_{12}^H G_{22}(y)  & -G_{12}^H G_{21}(y) 
								\am, \nonumber  \\
								\mathcal{C} :&=\left(\begin{array}{cc}
									\hat{\mu}_IG_{12}^HG_{11}^H &-\hat{\mu}_I(G_{11}^H)^2 
									\\ \\
									\hat{\mu}_I(G_{12}^H)^2 &-\hat{\mu}_IG_{12}^HG_{11}^H 
								\end{array}\right) . 
							\end{align}
							The term $\mathcal{A}(x)$ is a first order matrix-valued polynomial in $x$, $\mathcal{B}(y)$ is a first order matrix-valued polynomial in $y$ and $\mathcal{C}$ is a constant matrix. 
							We  define the Jost function as
							\begin{equation}\label{Jost function Markushevich}
								\mathcal{F}_\Theta (\xi):=\mathcal{F}'(0,\xi)+\Theta \mathcal{F}(0,\xi),
							\end{equation}
							with $\Theta$ as in \eqref{1rw(2.15)new}.
							After performing the Pekeris-Markushevich transform, it is easy to obtain analytical properties of the Jost function. However, all the results obtained in this framework need eventually to be converted back into the original framework using the following identities:
								\begin{align*}
									\mathcal{F}_\Theta (\xi) = \frac{1}{2 \xi \hat{\mu}_I \hat{\mu}(0)} \ma \hat{\mu}(0) & 0 \\ \\ 2\hat{\mu}_I \frac{\hat{\mu}'(0)}{\hat{\mu}(0)} & 2 \ii \hat{\mu}_I \xi \am \ma a(f_P) & a(f_S) \\ \\ b(f_P) & b(f_S) \am \xi \frac{\hat{\mu}_I}{\omega^2} \ma \ii & 0 \\ \\ 0 & -1 \am,
								\end{align*}
								where $B(\xi)$ is as in  \eqref{Jost matrix function}
								and
								\begin{align}\label{Connection B and F theta}
									&B(\xi) = A_1(\xi) \mathcal{F}_\Theta (\xi) A_2(\xi)
								\end{align}
								where 
								\begin{align}\label{Definition of A1}
									A_1(\xi)  =  \ma 2\xi \hat{\mu}_I & 0 \\ \\ 2 \ii \hat{\mu}_I \frac{\hat{\mu}'(0)}{\hat{\mu}(0)} & - \ii \hat{\mu}(0) \am, \qq 	A_2(\xi) = \ma \frac{\omega^2}{\ii \xi \hat{\mu}_I} & 0 \\ \\ 0 & -\frac{\omega^2}{\xi \hat{\mu}_I} \am.
								\end{align}

							\subsection{The Faddeev solution}
							In this subsection we introduce the Faddeev solutions as it is simpler to work with them rather than with the Jost solution because we get rid of some of the exponential factors that are present in the latter. We define the Faddeev solution to be
							\begin{align*}
								&H_P^\pm(x)= e^{\mp ix q_P} F_P^\pm,\qq H_S^\pm(x)= e^{\mp ix q_P} F_S^\pm,\\ 
								& H_{P,0}^\pm(x)= e^{\mp ix q_P} F_{P,0}^\pm,\qq H_{S,0}^\pm(x)= e^{\mp ix q_P} F_{S,0}^\pm,
							\end{align*} 
							and define $\mathcal{H} \equiv \mathcal{H}^+$, where 
							\begin{equation*}
								\mathcal{H}(x) = \ma  \left[H_{P}(x)\right]_1  &  \left[H_{S}(x)\right]_1 \\ \\  \left[H_{P}(x)\right]_2 &  \left[H_{S}(x)\right]_2 \am.
							\end{equation*}
							We consider the matrix composed of all the Faddeev solutions with $+$ sign\footnote{We can always swap to the other cases by applying the mappings \eqref{SP mappings} defined in the Riemann surface section.}.
							Then 
							\begin{align*}
								&H_P(x)=H_{P,0}(x)-\int_x^\infty \widetilde{\mathcal{G}}(x,y) V(y)H_P(y)dy,\\
								&H_{P,0}(x) = 
								\left(\begin{array}{c}
									G_{21}(x) + \ii q_P\frac{\hat{\mu}_I}{\omega^2}G_{11}^H 
									\\ \\
									G_{22}(x) + \ii q_P\frac{\hat{\mu}_I}{\omega^2}G_{12}^H
								\end{array}\right), \qq \widetilde{\mathcal{G}}(x,y)=e^{i(y-x)q_P} \mathcal{G}(x,y),
							\end{align*}
							\begin{align*}
								&H_S(x)=H_{S,0}(x)-\int_x^\infty \widetilde{\mathcal{G}}(x,y) V(y)H_S(y)dy, \qq H_{S,0}(x) =-\frac{\hat{\mu}_I \xi}{\omega^2}
								e^{\ii (q_S-q_P)x}\left(\begin{array}{c}
									G_{11}^H 
									\\ \\
									G_{12}^H
								\end{array}\right),
							\end{align*}
							so the unperturbed Faddeev solution is
							\begin{equation}\label{Unperturbed Faddeev solution}
								\mathcal{H}^+_0(x) =
								\left(\begin{array}{cc}
									G_{21}(x) + iq_P\frac{\hat{\mu}_I}{\omega^2}G_{11}^H  & -\frac{\hat{\mu}_I \xi}{\omega^2}
									G_{11}^H e^{-\ii (q_P - q_S)x} 
									\\ \\
									G_{22}(x) + iq_P\frac{\hat{\mu}_I}{\omega^2}G_{12}^H & -\frac{\hat{\mu}_I \xi}{\omega^2}
									G_{12}^H e^{-\ii (q_P - q_S)x} 
								\end{array}\right).
							\end{equation}
							The Volterra-type equation
							\begin{equation}\label{Volterra-type equation F}
								\mathcal{F}(x)=\mathcal{F}_{0} - \int_x^\infty \mathcal{G}(x,y) V(y)\mathcal{F}(y)dy,
							\end{equation}
							after multiplying both sides by $e^{-\ii xq_P }$ becomes
							\begin{align}
								&\mathcal{F}(x)e^{-\ii xq_P }= \mathcal{F}_{0} e^{-\ii xq_P } - \int_x^\infty \mathcal{G}(x,y) V(y) \mathcal{F}(y)e^{-\ii xq_P } dy, \nonumber
							\end{align}
							and in terms of the Faddeev solution
							\begin{align}
								\mathcal{H}(x) = \mathcal{H}_{0}(x)  - \int_x^\infty \mathcal{G}(x,y)e^{-\ii (x-y)q_P } V(y) \mathcal{H}(y)dy. \label{Volterra-type Faddeev}
							\end{align}
							We get a Volterra-type equation for the Faddeev solution $\mathcal{H}$ which we will use to derive the analytical properties of the Jost solution $\mathcal{F}$. From the Volterra-type equation we infer that
							\begin{align}\label{Expansion of Faddeev solution}
								\mathcal{H}(x) &= \mathcal{H}_{0}(x) - \int_x^\infty \widetilde{\mathcal{G}}(x,y) V(y)\mathcal{H}_0(y)dy + \int_x^\infty \widetilde{ \mathcal{G}} (x,y) V(y) \int_y^\infty  \widetilde{ \mathcal{G}} (y,t) V(t)  \mathcal{H}(t) dy dt \nonumber \\
								&=  \mathcal{H}_{0}(x, \xi) + \sum_{l=1}^{\infty} \mathcal{H}^{(l)}(x, \xi),
							\end{align}
							where for $l \geq 1$
							\begin{align}\label{Definition of Hl}
								\mathcal{H}^{(l)}(x, \xi) :=& (-1)^{l} \int_{x}^{H} \cdots \int_{t_{l-1}}^{H} \widetilde{\mathcal{G}}(x,t_1)  \cdots \widetilde{\mathcal{G}}(t_{l-1},t_l ) \cdot V(t_1) \cdots V(t_l) \mathcal{H}_0(t, \xi) dt_1 \cdots dt_l.
							\end{align}

							\subsection{Analytical properties of Jost solutions and Jost function}\label{Subsection Analytical properties}

							In this subsection, the goal is to obtain an asymptotic expansion of the entire function $F(\xi)$, as in Theorem \ref{Lemma asymptotic of F}, and an exponential type estimate as in Theorem \ref{F is of expponential type}. In order to achieve this, we need to find the asymptotic expansion of the determinants of the Jost function (see \eqref{Jost function Markushevich}) in the four different sheets and then convert the results into the framework before the Pekeris-Markushevich transform. This translates into asymptotics of the Rayleigh determinant $\Delta(\xi)$ in the four different sheets and eventually into asymptotics of $F(\xi)$ (see \eqref{equation entire funct F}).

							First, we define a class of potentials for which all the following results will hold.

							\begin{definition}[Class of potentials]\label{Class of potentials Rayleigh}
								We denote by $\mathcal{V}_{H}$ the class of $V$ such that $V \in  L^1(\mathbb{R}_+; \mathbb{C}^{2 \times 2})$, continuous and $\supp V \subset \left[0,H\right]$ for some $H >0$ and for each $\epsilon >0$ the set $\left(H-\epsilon, H\right) \cap \supp V_{ij}$, for $i,j=1,2$, has positive Lebesgue measure. 
							\end{definition}
							For such class of potentials, we have the following results.
							
							\begin{theorem}[Jost solutions]\label{th-Jost-solutions}
								For $V \in \mathcal{V}_{H}$ and any fixed $x\geq 0,$ the Jost solution $\mathcal{F}(x,\xi)$ is analytic in $\xi$ on each sheet $\Xi_{\pm,\pm}$, of exponential type,  and for $\xi \in \Xi$ satisfying
								\begin{align}\label{Estimates on Markushevich Jost solution}
									& \mathcal{F}(x,\xi)\equiv \mathcal{F}^+(x,\xi)=\mathcal{F}_0^+(x,\xi)-\int_x^H \mathcal{G}(x,y) V(y) \mathcal{F}_0^+(y,\xi)dy+\sum_{k=2}^\infty \mathcal{F}_k(x,\xi), \nonumber \\
									& ||\mathcal{F}_k(x,\xi)|| \leq C \frac{|\xi|}{k!}e^{\gamma(\xi) (H-x)}e^{H \frac{\zeta_{P-S} }{2}} e^{-x\left(\frac{\zeta_P}{2}\right)} \left(\mathfrak{a}(x) \right)^k, 
								\end{align}
								where
								\begin{equation*}
									\zeta_{P-S} := \im (q_P - q_S) + |\im (q_P - q_S)|, \qq \zeta_P := \im q_P + |\im q_P|,
								\end{equation*}
								\begin{equation*}
									\mathfrak{a}(x) := \frac{\int_x^H ||V(t)|| dt}{\max\{1,|\xi|\}},
								\end{equation*}
								\begin{align}\label{gamma of xi}
									\gamma(\xi) =
									\begin{cases}
										0 \qqq \qqq &\text{for }\xi \in \Xi_{+, \pm} \\
										-2\im q_P &\text{for }\xi \in \Xi_{-, \pm}
									\end{cases}.
								\end{align}
							\end{theorem}
							\begin{proof}
A lenghty but straightforward computation reveals that
								\begin{align*}
									\widetilde{\mathcal{G}}(x,y) = \mathcal{G}(x,y)e^{-\ii (x-y)q_P } &= \mathcal{A}(x)  \left[  \frac{1 - e^{2 \ii q_P(y-x)}}{2 \ii q_P} \right] + \mathcal{B}(y) \left[ \frac{-e^{\ii (y-x)(q_P + q_S)} + e^{\ii (y-x)(q_P - q_S)}}{2 \ii q_S} \right] \\
									&\phantom{=\;}+\mathcal{C} \left[  \frac{ - e^{2 \ii (y-x)q_P} + e^{\ii (y-x)(q_P + q_S)} + e^{\ii (y-x) (q_P - q_S)}  - 1 }{2\omega^2} \right].
								\end{align*}
								We take the maximum norm of $\widetilde{\mathcal{G}}(x,y, \xi) $, which is the maximum\footnote{Note that the $\xi$ dependence is implicit inside the quasi momenta $q_S(\xi)$ and $q_P(\xi)$} for fixed $x$, $y$ and $\xi$ of the absolute value of the components of $\widetilde{\mathcal{G}}(x,y, \xi) $. Taking into account that $0\leq x \leq y \leq H$, which implies $y-x \geq 0$, 
								we get
								\begin{equation*}
									|| \widetilde{\mathcal{G}}(x,y,\xi)|| \leq \frac{C}{\max\{1,|\xi|\}} e^{ (y-x) \gamma(\xi)}
								\end{equation*}
								where $C>0$ is a constant which does not depend on $\xi$, $x$ and $y$, and where $\gamma(\xi)$ is defined as 
								\begin{align}\label{gamma definition}
									&\gamma(\xi) = \max \left\lbrace |\im q_P| - \im q_P, \frac{|\im (q_P - q_S)| - \im (q_P -q_S)}{2} , \frac{|\im (q_P + q_S)| - \im (q_P + q_S)}{2}  \right\rbrace.
								\end{align}
								From \eqref{gamma definition}, we can see that $\gamma(\xi) = 0$ in $\Xi_{+, \pm}$, because $\Im q_P >0$, $\im (q_P + q_S) >0$ and $\im (q_P - q_S) >0$ (see Lemma \ref{Lemma sheets of Riemann surface}). While in the sheets $\Xi_{-, \pm}$, $\gamma(\xi) = -2 \im q_P$ as $\Im q_P >0$, $\im (q_P + q_S) <0$ and $\im (q_P - q_S) <0$ (see Lemma \ref{Lemma sheets of Riemann surface}) which imply
								\begin{align}
									& \im (q_P - q_S) < 0 \iff 	\im q_P < \im q_S  \iff 2\im q_P < \im q_P + \im q_S &\iff -2\im q_P > -\im (q_P + q_S),
								\end{align}
								and
								\begin{align*}
									& \im (q_P + q_S) < 0 \iff \im q_P < - \im q_S \iff -\im q_P > \im q_S \iff -2\im q_P >-\im (q_P - q_S)
								\end{align*}
								leading to formula \eqref{gamma of xi}.
								From \eqref{Unperturbed Faddeev solution} we can calculate the norm of the unperturbed Faddeev solution $\mathcal{H}_0(x)$, which is
								\begin{equation*}
									|| \mathcal{H}_0(x, \xi)|| \leq C |\xi| e^{x \frac{\zeta_{P-S}}{2}}.
								\end{equation*}
								Then, the norm of \eqref{Volterra-type Faddeev} after one iteration can be estimated as follows:
								\begin{equation*}
									||\mathcal{H}(x, \xi)|| \leq ||\mathcal{H}_{0}(x, \xi)||  + \int_x^\infty  \frac{C}{\max\{1,|\xi|\}} e^{ (y-x) \gamma(\xi)} ||V(y)|| \; ||\mathcal{H}(y, \xi)||dy.
								\end{equation*}
								Starting from \eqref{Volterra-type Faddeev} and iterating the equation we get the series 
								\begin{align*}
									\mathcal{H} (x,\xi)  =   \sum_{l=0}^{\infty} \mathcal{H}^{(l)} (x,\xi),
								\end{align*}
								where
								\begin{align*}
									\mathcal{H}^{(0)} (x,\xi)  = \mathcal{H}_{0} (x,\xi) 
								\end{align*}
								and any $l$-term is uniformly bounded by
								\begin{align*}
									&|| \mathcal{H}^{(l)} (x,\xi) || \leq \int_x^\infty \int_{t_1}^\infty \cdots \int_{t_{l-1}}^\infty C \frac{e^{\gamma(\xi) \left[t_1-x + (t_2 - t_1) + \cdots (t_l - t_{l-1})\right]}}{\left( \max\{1,|\xi|\}\right)^l} ||V(t_1)|| \cdots ||V(t_l)|| \; ||\mathcal{H}_0 (t_l,\xi)|| dt_l \cdots dt_1  \\
									&=C  |\xi| e^{H \frac{\zeta_{P-S}}{2}} \int_x^\infty \int_{t_1}^\infty \cdots \int_{t_{l-1}}^\infty \frac{e^{\gamma(\xi) (t_l-x)}}{\left( \max\{1,|\xi|\} \right)^l} ||V(t_1)|| \cdots ||V(t_l)||  dt_l \cdots dt_1 \\
									&\leq C \frac{e^{\gamma(\xi) (H-x)}}{\left( \max\{1,|\xi|\} \right)^l}  |\xi| e^{H \frac{\zeta_{P-S}}{2}} \frac{1}{l!} \left(\int_x^H ||V(t)|| dt\right)^l  \leq  C \frac{e^{\gamma(\xi) H}}{\left( \max\{1,|\xi|\} \right)^l}  |\xi| e^{H \frac{\zeta_{P-S}}{2}} \frac{1}{l!} \left(\int_0^H ||V(t)|| dt\right)^l 
								\end{align*}
								with the convention $t_0=x$ and since $V(t) = 0$ for $t\geq H$. Hence the series is absolutely and uniformly convergent on every compact set. Then the Faddeev solution $\mathcal{H}^{(l)} (x,\xi)$ is analytic in each sheet $\Xi_{\pm,\pm}$.
							\end{proof}
							
							For $\xi$ in the physical sheet, the Faddeev solution is complex analytic, hence continuous.
							We can see that
							\begin{align*}
								& ||\mathcal{H}(x, \xi) - \mathcal{H}_{0}(x, \xi)|| \leq \sum_{l=1}^{\infty} || \mathcal{H}^{(l)} (x,\xi) || \leq  |\xi| e^{\gamma(\xi) (H-x)} e^{H \frac{\zeta_{P-S}}{2}} e^{\frac{\int_x^H ||V(t)|| dt}{\max\{1,|\xi|\}}}.
							\end{align*}
							Then, if we want to obtain those estimates in terms of the Jost solution $\mathcal{F}(x)$, we have
							\begin{align*}
								||\mathcal{H}(x, \xi) - \mathcal{H}_{0}(x, \xi)|| = ||\mathcal{F}(x, \xi)e^{-ixq_P} - \mathcal{F}_{0}(x, \xi)e^{-ixq_P}|| = ||\mathcal{F}(x, \xi)- \mathcal{F}_{0}(x, \xi)|| e^{x\left(\frac{\zeta_P}{2}\right)}
							\end{align*}
							which leads to 
							\begin{align*}
								||\mathcal{F}(x, \xi)- \mathcal{F}_{0}(x, \xi)|| \leq |\xi| e^{\gamma(\xi) (H-x)} e^{H \frac{\zeta_{P-S}}{2}} e^{-x\left(\frac{\zeta_P}{2}\right)} e^{\frac{\int_x^H ||V(t)|| dt}{\max\{1,|\xi|\}}}
							\end{align*}
							as in \eqref{Estimates on Markushevich Jost solution}.

							A theorem similar to Theorem \ref{th-Jost-solutions} can  also be found in \cite{Iantchenko} but for $\xi$ only in the physical sheet, with a different Riemann surface and without proof.
						
							\begin{remark}
								We get different estimates on the Jost solution depending on the sheet, in particular
								\begin{itemize}
									\item \textit{Sheets $\Xi_{+,+}$, $\Xi_{+,-}$.} We have that $\im (q_P + q_S) >0$, $\im (q_P - q_S) >0$ and $\Im q_P >0$ which imply that $\gamma(\xi) = 0$ and 
									\begin{equation}\label{Estimates physical sheet}
										||\mathcal{F}(x, \xi)- \mathcal{F}_{0}(x, \xi)|| \leq |\xi|  e^{(H-x)\im q_P} e^{-H \im q_S}e^{\frac{\int_x^H ||V(t)|| dt}{\max\{1,|\xi|\}}}.
									\end{equation}
									
									\item \textit{Sheets $\Xi_{-,-}$, $\Xi_{-,+}$.} 
									We have that $\im (q_P + q_S) <0$, $\im (q_P - q_S) <0$ and $\Im q_P <0$ which imply that $\gamma(\xi) = -2 \im q_P$, hence
									\begin{equation}\label{Estimates first unphysical sheet}
										||\mathcal{F}(x, \xi)- \mathcal{F}_{0}(x, \xi)|| \leq |\xi|  e^{-2(H-x)\im q_P} e^{\frac{\int_x^H ||V(t)|| dt}{\max\{1,|\xi|\}}}.
									\end{equation}
								\end{itemize}
							\end{remark}

							In the next lemma we show a similar result to Theorem \ref{th-Jost-solutions} for the derivative of the Jost solution.
							
							\begin{lemma}\label{Estimate Derivative Markushevich Jost solution}
								For any fixed $x\geq 0,$ the derivative of the Jost solution $\mathcal{F}'(x,\xi):=\frac{\partial}{\partial x}\mathcal{F}(x,\xi)$ is analytic in $\xi$ in each sheet $\Xi_{\pm, \pm}$ and satisfies
								\begin{align}\label{Estimates on derivative of Markushevich Jost solution}
	||\mathcal{F}'(x,\xi)||  \leq   C   |q_S||q_P||\xi|  \exp \left( \gamma(\xi) (H-x) + H \frac{\zeta_{P-S}}{2} -x\frac{\zeta_P}{2} + \frac{\int_x^H ||V(t)|| dt }{\max\{1,|\xi|\}} \right).
								\end{align}
							\end{lemma}
							
							\begin{proof}
								To find an estimate for the derivative of the Jost solution, we differentiate \eqref{Volterra-type equation F} with respect to $x$ and we get
								\begin{align}\label{Volterra type for the derivative of F}
									\mathcal{F}'(x,\xi)&=\ii \mathcal{F}_{0}(x,\xi) \mathcal{Q}(\xi) +  \mathcal{P}(x,\xi) + \mathcal{G}(x,x) V(x)\mathcal{F}(x, \xi) - \int_x^\infty \mathcal{G}_x(x,y) V(y)\mathcal{F}(y, \xi)dy \nonumber \\
									&= \ii \mathcal{F}_{0}(x,\xi) \mathcal{Q}(\xi) + \mathcal{P}(x,\xi) - \int_x^\infty \mathcal{G}_x(x,y) V(y)\mathcal{F}(y, \xi)dy,
								\end{align}
								where
								\begin{align*}
									& \mathcal{F}'_0(x,\xi) \mathcal{Q}(\xi) + \mathcal{P}(x,\xi)= \ma \left( \frac{c_I}{2}G_{11}^H H + \ii q_P  G_{21}(x)- q_P^2 \frac{\hat{\mu}_I}{\omega^2} G_{11}^H \right) e^{\ii x q_P} & - \frac{\ii q_S \hat{\mu}_I \xi}{\omega^2}G_{11}^H e^{\ii x q_S} \\ \\
									\left( \frac{c_I}{2}G_{12}^H H + \ii q_P G_{22}(x) - q_P^2 \frac{\hat{\mu}_I}{\omega^2} G_{12}^H \right) e^{\ii x q_P} & - \frac{\ii q_S \hat{\mu}_I \xi}{\omega^2}G_{12}^H e^{\ii x q_S} \am,
								\end{align*}
								with
								\begin{equation*}
									\mathcal{Q}(\xi) = \ma q_P & 0  \\ \\
									0  & q_S \am, \qquad \mathcal{P}(x,\xi) = \ii q_P \frac{c_I}{2} H e^{\ii x q_P }\ma G_{11}^H & 0  \\ \\
									G_{12}^H  & 0 \am \qquad c_I := \frac{\hat{\lambda}_I + \hat{\mu}_I}{\hat{\lambda}_I + 2 \hat{\mu}_I}.
								\end{equation*}
								In the first passage of \eqref{Volterra type for the derivative of F} we used the property \eqref{Conditions for Green kernel} of the Green kernel which made the second term zero. The $x$ derivative of the Green kernel is
								\begin{align*}
									\mathcal{G}_x(x,y) &=  \mathcal{A}_0\frac{\sin((x-y)q_P)}{q_P} + \mathcal{A}(x) \cos((x-y)q_P)+ \mathcal{B}(y) \cos((x-y)q_S)  \\
									&\phantom{=\;}+ \mathcal{C}\frac{ q_P\sin((x-y) q_P) - q_S\sin( (x-y)q_S)}{\omega^2},
								\end{align*}
								where 
								\begin{equation*}
									\mathcal{A}_0 = \frac{c_I}{2}\ma G_{11}^H G_{12}^H & -\left[G_{11}^H \right]^2 \\ \\ \left[G_{11}^H \right]^2 & -G_{11}^H G_{12}^H \am.
								\end{equation*}
								As for the Jost solution, if we want to estimate $\widetilde{ \mathcal{G}}_x (x,y) := \mathcal{G}_x(x,y) e^{\ii (x-y) q_P}$, we need to look at all the sine and cosine terms:
								\begin{align*}
									&\frac{ \sin \left[q_P(x-y)\right]}{q_P} e^{-\ii (x-y)q_P } = \frac{1 - e^{2 \ii q_P(y-x)}}{2 \ii q_P}; \\
									&\cos((x-y)q_P)e^{-\ii (x-y)q_P } =  \frac{1 + e^{2 \ii q_P(y-x)}}{2 };  \\
									& \cos((x-y)q_S)e^{-\ii (x-y)q_P } = \frac{e^{\ii (y-x)(q_P - q_S)} + e^{\ii (y-x)(q_P + q_S)}}{2}; \\
									& \left[ q_P\sin((x-y) q_P) - q_S\sin( (x-y)q_S)\right]e^{-\ii (x-y)q_P }  \\
									& =\frac{1}{2\ii} \left[ q_P\left( 1 - e^{2 \ii q_P(y-x)}\right) - q_S \left( e^{\ii (y-x)(q_P - q_S)} - e^{\ii (y-x)(q_P + q_S)} \right) \right].
								\end{align*}
								These expressions imply the following estimate on the  $\widetilde{ \mathcal{G}}_x (x,y) $
								\begin{equation*}
									||\widetilde{ \mathcal{G}}_x (x,y)|| \leq C |q_S| e^{\gamma(\xi) (y-x)}.
								\end{equation*}
								Multiplying both sides of \eqref{Volterra type for the derivative of F} by $e^{-\ii (x-y)q_P }$, iterating the Faddeev solution $\mathcal{H}$ and taking the norm we find
								\begin{align*}
									||\mathcal{F}'(x,\xi) e^{-\ii xq_P }|| &\leq || \left( \ii \mathcal{F}_0 (x,\xi) \mathcal{Q}(x) +\mathcal{P}(x,\xi) \right) e^{-\ii x q_P } || + \sum_{l=1}^{\infty} || \mathcal{M}^{(l)}(x,\xi) ||; \nonumber \\ 
									||\mathcal{M}^{(l)}(x,\xi)|| &\leq C \frac{e^{\gamma(\xi) (H-x)}}{\left( \max\{1,|\xi|\}\right)^{l-1}}  |q_S||\xi| e^{H \frac{\zeta_{P-S}}{2}} e^{-x\left(\frac{\zeta_P}{2}\right)}   \frac{1}{l!} \left(\int_x^H ||V(t)|| dt\right)^l \\
									&= C \frac{e^{\gamma(\xi) (H-x)}}{\left( \max\{1,|\xi|\}\right)^{l}}  |q_S||q_P||\xi| e^{H \frac{\zeta_{P-S}}{2}}  e^{-x\left(\frac{\zeta_P}{2}\right)} \frac{1}{l!} \left(\int_x^H ||V(t)|| dt\right)^l,
								\end{align*}
								so we recover \eqref{Estimates on derivative of Markushevich Jost solution}. $\mathcal{M}(x,\xi)$ is bounded by a uniformly convergent series, then $\mathcal{F}'(x,\xi) $ is analytic in each sheet $\Xi_{\pm,\pm}$.
							\end{proof}
							
							From Theorem \ref{th-Jost-solutions} and Lemma \ref{Estimate Derivative Markushevich Jost solution} we can obtain estimates of the Jost function and of the entire function $F(\xi)$, defined in Section \ref{Section F entire}, which shows that $F(\xi)$ is of exponential type. We present this result in the following lemma.

							\begin{corollary}\label{Corollary Exp type of F}
								The function $F(\xi)$ is of exponential type with order one and type at most $12H$. In particular, for $|\xi|\to \infty$ it satisfies the inequality
								\begin{align}\label{Exponential type of function F}
									|F(\xi)| \leq C  |\xi|^{20} e^{12H|\re \xi|}.
								\end{align}
							\end{corollary}
							\begin{proof}
								From Theorem \ref{th-Jost-solutions} we see that the Jost solution at $x=0$ is of exponential type 
								\begin{align}\label{Exponential type estimate}
									||\mathcal{F}(0,\xi)|| \leq C |\xi|  e^{2H \mathfrak{b}(\xi)}
								\end{align}
								where 
								\begin{align*}
									\mathfrak{b}(\xi) := 
									\begin{cases}
										0 \qqq \qq &\xi \in \Xi_{+,+} \\
										\re \xi  &\xi\in \Xi \setminus  \Xi_{+,+}
									\end{cases}.
								\end{align*}
								The inequality \eqref{Exponential type estimate} is obtained from \eqref{Quasi momenta in the 4 sheets}, Theorem \ref{th-Jost-solutions} and Lemma \ref{Lemma sheets of Riemann surface} that imply
								\begin{align*}
									\gamma(\xi) =
									\begin{cases}
										0 \qqq \qqq &\text{for }\xi \in \Xi_{+, \pm} \\
										-2\im q_P &\text{for }\xi \in \Xi_{-, \pm},
									\end{cases}
								\end{align*}
								and
								\begin{align*}
									\zeta_{P-S} =
									\begin{cases}
										2  \im (q_P - q_S) \qqq \qqq &\text{for }\xi \in \Xi_{+, \pm} \\
										0 &\text{for }\xi \in \Xi_{-, \pm}
									\end{cases}.
								\end{align*}
								From Lemma \ref{Estimate Derivative Markushevich Jost solution}, we can see that it holds
								\begin{align*}
									||\mathcal{F}'(0, \xi)|| \leq C |\xi|^2  e^{2H \mathfrak{b}(\xi)},
								\end{align*}
								and since $\mathcal{F}_{\Theta}( \xi):= \mathcal{F}'(0, \xi) + \Theta \mathcal{F}(0, \xi)$ with 
								\begin{equation*}
									\Theta = \xi^2 \frac{2\hat{\mu}_I}{\hat{\mu}(0)} \ma 0 & 0 \\ \\ 1 & 0 \am + \ma -\theta_3 &\theta_2\\ \\ - \theta_1 & 0 \am,
								\end{equation*}
								we have
								\begin{align*}
									||\mathcal{F}_{\Theta}(\xi)|| \leq C |\xi|^3  e^{2H \mathfrak{b}(\xi)}.
								\end{align*}
								We know that $\Delta(\xi) = \det \left( A_1(\xi) \mathcal{F}_{\Theta}(\xi))  A_2(\xi) \right)$ and from \eqref{Definition of A1} we get
								\begin{align*}
									||\Delta(\xi) || \leq C |\xi|^5 \qqq \xi \in \Xi_{+,+}
								\end{align*}
								and 
								\begin{align*}
									||\Delta(w_{\bullet}(\xi)) || \leq C |\xi|^5 e^{4H \re \xi}, \qq \bullet  = P,S,PS.
								\end{align*}
								Since $F(\xi) := \Delta(\xi) \Delta(w_{P}(\xi)) \Delta(w_{S}(\xi)) \Delta(w_{PS}(\xi))$, we have an estimate on the entire function $F(\xi)$:
								\begin{equation*}
									||F(\xi)) || \leq C |\xi|^{20} e^{12H \re \xi}, \qqq \xi \in \mathbb{C}.
									\qedhere
								\end{equation*}
							\end{proof}
							
							\begin{remark}
								In Corollary \ref{Corollary Exp type of F} the term $\xi^{20}$ and the power of the exponential $12H \re \xi$ are not sharp estimates as we could have cancellations in the computation of the determinant. For example, from the definition of $\Theta$ we can see that the second row is of order $\xi^2$ while the first row is of lower order, so the determinant of $\mathcal{F}_{\Theta}(\xi)$ can never be, say, of polynomial order $\xi^{4}$. At the end of the section, we will obtain a sharp estimate on the type of the exponential.
							\end{remark}

							\subsection{Estimates of the Jost solution and Jost function}\label{Subsection Estimates Jost}

							The goal is to obtain an asymptotic expansion of the terms in the formula \eqref{Expansion of Faddeev solution}. Therefore, we compute the asymptotic expansion of $\mathcal{H}_0(x,\xi)$ in Lemma \ref{Lemma asymptotics of H0}, then the asymptotics of the term $\int_x^\infty \widetilde{\mathcal{G}}(x,y) V(y)\mathcal{H}_0(y)dy$ in Lemma \ref{Lemma asymptotics of first iterate of H} and the asymptotics of the second iterate of the Volterra equation in \eqref{Expansion of Faddeev solution} in Lemma \ref{Lemma asymptotics of second iterate of H}. From the result of these three lemmas we obtain the asymptotic expansion of the Jost solution in Lemma \ref{Lemma asymptotics on Jost solution} and of the Jost function in Proposition \ref{Proposition Jost func in phys sheet}.
							We can simplify the notation by defining
							\begin{equation*}
								G^H := \ma G_{11}^H & G_{11}^H  \\ \\ G_{12}^H  & G_{12}^H \am,  \qquad \qquad 	G_H(y) := \ma G_{21}(y) & -\frac{c_I y}{2} G_{11}^H \\ \\  G_{22}(y) & -\frac{c_I y}{2} G_{12}^H \am
							\end{equation*}
							which appear very often in the following.
							\begin{lemma}\label{Lemma asymptotics of H0}
								The unperturbed Faddeev solution $\mathcal{H}_0(x,\xi) $ for $|\xi| \to \infty$ and $\xi$ on the physical sheet $\Xi_{+,+}$ admits the asymptotic expansion
								\begin{align}\label{asymptotics of H0}
									\mathcal{H}_0(x,\xi) =& -\xi  \frac{\hat{\mu}_I}{\omega^2} G^H + G_H(x)  +\xi^{-1} \ma \frac{\hat{\mu}_I}{2(\hat{\lambda}_I + 2\hat{\mu}_I)} G_{11}^H  & - \frac{\omega^2 c^2_{I} x^2}{8 \hat{\mu}_I } G_{11}^H  \\ \\ \frac{\hat{\mu}_I}{2(\hat{\lambda}_I + 2\hat{\mu}_I)} G_{12}^H  & - \frac{\omega^2 c^2_{I} x^2}{8 \hat{\mu}_I } G_{12}^H  \am + \mathcal{O}(|\xi|^{-2}).
								\end{align}
							\end{lemma}
							\begin{proof}
								The unperturbed Faddeev solution can be written as 
								\begin{equation*}
									\mathcal{H}^{\pm}_0(x,\xi) =
									\left(\begin{array}{cc}
										G_{21}(x) \pm iq_P\frac{\hat{\mu}_I}{\omega^2}G_{11}^H  & -\frac{\hat{\mu}_I \xi}{\omega^2}
										G_{11}^H e^{-\ii (q_P - q_S)x} 
										\\ \\
										G_{22}(x) \pm iq_P\frac{\hat{\mu}_I}{\omega^2}G_{12}^H & -\frac{\hat{\mu}_I \xi}{\omega^2}
										G_{12}^H e^{-\ii (q_P - q_S)x} 
									\end{array}\right).
								\end{equation*}
								When  $|\xi| \to \infty$ and $\xi \in \Xi_{+,+}$, we can expand the quasi-momenta in powers of $\xi$, namely $q_S= \ii \xi - \ii \frac{{\omega^2}}{2\hat{\mu}_I \xi } + \mathcal{O}(|\xi|^{-3})$ and $q_P =  \ii \xi - \ii \frac{\omega^2}{2(\hat{\lambda}_I + 2\hat{\mu}_I) \xi } + \mathcal{O}(|\xi|^{-3})$ which imply
								\begin{align*}
									&q_P - q_S= \frac{\ii \omega^2 c_I}{2 \xi} + \mathcal{O}(|\xi|^{-3}), & e^{-\ii (q_P - q_S)x} = 1 + \frac{\omega^2 c_Ix}{2 \hat{\mu}_I \xi} + \frac{\omega^4 c^2_{I} x^2}{8 \hat{\mu}^2_I \xi^2} + \mathcal{O}(\xi^{-3}).
								\end{align*}
								Hence, plugging the asymptotic expansions into the definition of the unperturbed Faddeev solution we get an asymptotic expansion in powers of $\xi$
								
								\begin{align*}
									\mathcal{H}_0(x,\xi) &=  \ma G_{21}(x) - \frac{\hat{\mu}_I G_{11}^H}{\omega^2} (   \xi - \frac{\omega^2}{2\sigma_I \xi} + \mathcal{O}(\xi^{-3}) ) & -\frac{\hat{\mu}_I G_{11}^H \xi}{\omega^2}
									(  1 + \frac{\omega^2 c_Ix}{2 \hat{\mu}_I \xi} + \frac{\omega^4 c^2_{I} x^2}{8 \hat{\mu}^2_I \xi^2} + \mathcal{O}(\xi^{-3}) ) \\ \\  G_{22}(x) - \frac{\hat{\mu}_IG_{12}^H}{\omega^2} (  \xi - \frac{\omega^2}{2\sigma_I \xi}+ \mathcal{O}(\xi^{-3}) ) & -\frac{\hat{\mu}_I G_{12}^H \xi}{\omega^2}
									(  1 + \frac{\omega^2 c_Ix}{2 \hat{\mu}_I \xi} + \frac{\omega^4 c^2_{I} x^2}{8 \hat{\mu}^2_I \xi^2} + \mathcal{O}(\xi^{-3}) ) \am \\
									&= -\xi  \frac{\hat{\mu}_I}{\omega^2} \ma G_{11}^H & G_{11}^H  \\ \\ G_{12}^H  & G_{12}^H \am + \ma G_{21}(x) & -\frac{c_I G_{11}^H x}{2}  \\ \\  G_{22}(x) & -\frac{c_I G_{12}^H x}{2}  \am  + \xi^{-1} \ma \frac{\hat{\mu}_I G_{11}^H}{2(\hat{\lambda}_I + 2\hat{\mu}_I)}   & - \frac{\omega^2 c^2_{I} G_{11}^H x^2}{8 \hat{\mu}_I }  \\ \\ \frac{\hat{\mu}_IG_{12}^H }{2(\hat{\lambda}_I + 2\hat{\mu}_I)}  & - \frac{\omega^2 c^2_{I} G_{12}^H x^2}{8 \hat{\mu}_I }   \am  + \mathcal{O}(|\xi|^{-2})
								\end{align*}
								with $\sigma_I:= \hat{\lambda}_I + 2 \hat{\mu}_I$.
							\end{proof}

							The goal is to obtain an asymptotic expansion of the Jost solution and the Jost function. We compute this from the asymptotic expansion of the Faddeev solution. In the previous lemma we obtained the asymptotic expansion of the first term $\mathcal{H}_0$ in \eqref{Expansion of Faddeev solution}. In the next lemma we compute the asymptotic expansion of the first Volterra iterate $\mathcal{H}^{(1)}(x,\xi)$ in \eqref{Expansion of Faddeev solution}.

							\begin{lemma}\label{Lemma asymptotics of first iterate of H}
								For $V \in \mathcal{V}_{H}$, the first Volterra iterate $\mathcal{H}^{(1)}(x,\xi)$ of \eqref{Expansion of Faddeev solution} for $|\xi| \to \infty$ and $\re \xi > 0$ in the physical sheet $\Xi_{+,+}$ admits the asymptotic expansion
								\begin{align*}
								\int_{x}^{H} \widetilde{ \mathcal{G}}(x,y)V(y) \mathcal{H}_0(y,\xi) dy &= \int_{x}^{H} \frac{\hat{\mu}_I}{2\omega^2} V(y) G^H dy \\
									&\phantom{=\;}- \frac{1}{4 \xi} \frac{\hat{\mu}_I}{\omega^2}  \int_{x}^{H} \left(  \mathcal{B}(y) \frac{ \omega^2 c_I (y-x)}{\hat{\mu}_I} + \mathcal{C} \frac{\omega^2 c^2_{I} (y-x)^2}{4 \hat{\mu}^2_I} \right) V(y) G^H dy\\
									&\phantom{=\;} - \frac{1}{2 \xi} \int_{x}^{H} V(y) G_H(y) dy - \frac{1}{4 \xi} \frac{\hat{\mu}_I}{\omega^2}  V(x) G^H + \mathcal{O}(\xi^{-2}).
								\end{align*}
							\end{lemma}
							\begin{proof}
								We know that $\widetilde{ \mathcal{G}}(x,y)$ is the transformed kernel defined as
								\begin{align*}
									&\widetilde{\mathcal{G}}(x,y) =  \widetilde{\mathcal{G}}_1(x,y) + \widetilde{\mathcal{G}}_2(x,y),
								\end{align*}
								with
								\begin{align*}
									\widetilde{\mathcal{G}}_1(x,y) := \frac{ \mathcal{A}(x) }{2 \ii q_P} + \mathcal{B}(y) \frac{ e^{\ii (y-x)(q_P - q_S)}}{2 \ii q_S} + \mathcal{C}  \frac{e^{\ii (y-x) (q_P - q_S)}}{2\omega^2} - \frac{\mathcal{C} }{2 \omega^2}
								\end{align*}
								and 
								\begin{align*}
									\widetilde{\mathcal{G}}_2(x,y) :=& - \frac{ \mathcal{A}(x) }{2 \ii q_P} e^{2 \ii q_P(y-x)}  - \mathcal{B}(y) \frac{ e^{\ii (y-x)(q_P + q_S)}}{2 \ii q_S} \nonumber \\
									&+ \mathcal{C}  \frac{e^{\ii (y-x) (q_P + q_S)}}{2\omega^2} - \frac{\mathcal{C} }{2 \omega^2} e^{2 \ii q_P(y-x)}.
								\end{align*}
								We can divide the proof into three steps: in the first step we compute the contribution to the integral $\int_{x}^{H} \widetilde{ \mathcal{G}}(x,y)V(y) \mathcal{H}_0(y,\xi) dy$ given by $\widetilde{\mathcal{G}}_1(x,y)$; in the second step we calculate the whole contribution coming from the term $\widetilde{\mathcal{G}}_2(x,y) $; and in the third step we sum up the results.
								\begin{itemize}

									\item \textit{Step 1.} When $|\xi| \to \infty$ and $\re \xi > 0$ on the physical sheet $\Xi_{+,+}$, we can use the expansion of the quasi-momenta in terms of powers of $\xi$ and we get the following form of $\widetilde{\mathcal{G}}_1(x,y)$
									\begin{align*}
										&\widetilde{\mathcal{G}}_1(x,y) = \frac{\mathcal{A}(x)}{-2\xi} \left( 1 + \mathcal{O}(\xi^{-2})\right) + \frac{\mathcal{B}(y) }{-2\xi} \left(1 + \mathcal{O}(\xi^{-2}) \right) \left( 1 - \frac{\omega^2 c_I(y-x)}{2 \hat{\mu}_I \xi} + \mathcal{O}(\xi^{-2}) \right) \\
										&\phantom{\;=} +  \frac{\mathcal{C}}{2\omega^2} \left( 1 - \frac{\omega^2 c_I(y-x)}{2 \hat{\mu}_I \xi} + \frac{\omega^4 c^2_{I} (y-x)^2}{8 \hat{\mu}^2_I \xi^2} + \mathcal{O}(\xi^{-3}) \right) - \frac{\mathcal{C} }{2 \omega^2} \\
										&= \frac{1}{-2 \xi} \left( \mathcal{A}(x) + \mathcal{B}(y) + \frac{c_I (y-x)}{2 \hat{\mu}_I} \mathcal{C} \right) + \frac{1}{4 \xi^2} \left(  \mathcal{B}(y) \frac{ \omega^2 c_I (y-x)}{\hat{\mu}_I} + \mathcal{C} \frac{\omega^2 c^2_{I} (y-x)^2}{4 \hat{\mu}^2_I} \right)  + \mathcal{O}(\xi^{-3}) \\
										&= \frac{1}{-2 \xi} + \frac{1}{4 \xi^2} \left(  \mathcal{B}(y) \frac{ \omega^2 c_I (y-x)}{\hat{\mu}_I} + \mathcal{C} \frac{\omega^2 c^2_{I} (y-x)^2}{4 \hat{\mu}^2_I} \right) + \mathcal{O}(\xi^{-3})
									\end{align*}
									where we can simplify
									\begin{align}\label{A + B + C}
										&\mathcal{A}(x) + \mathcal{B}(y) + \frac{c_I (y-x)}{2 \hat{\mu}_I} \mathcal{C} =\left(  G_{11}^H G_{22}^H - G_{12}^HG_{21}^H\right)  \ma 1 & 0 \\ \\ 0 & 1 \am = \det G \ma 1 & 0 \\ \\ 0 & 1 \am = \ma 1 & 0 \\ \\ 0 & 1 \am
									\end{align}
									using \eqref{Definition of A B and C}.
									Then the integral term can be written as
									
									\begin{align}\label{Integral of G1}
										\int_{x}^{H} \widetilde{ \mathcal{G}}_1(x,y)V(y) \mathcal{H}_0(y,\xi) dy &= \int_{x}^{H} \frac{\hat{\mu}_I}{2\omega^2} V(y) G^H dy \nonumber \\
										&\phantom{=\;} - \frac{1}{4 \xi} \frac{\hat{\mu}_I}{\omega^2}  \int_{x}^{H} \left(  \mathcal{B}(y) \frac{ \omega^2 c_I (y-x)}{\hat{\mu}_I} + \mathcal{C} \frac{\omega^4 c^2_{I} (y-x)^2}{4 \hat{\mu}^2_I} \right) V(y) G^H dy \nonumber \\
										&\phantom{=\;} - \frac{1}{2 \xi} \int_{x}^{H} V(y) G_H(y) dy + O\left(\frac{1}{|\xi|^2}\right).
									\end{align}
									\item \textit{Step 2.}
									In $\widetilde{\mathcal{G}}_2(x,y)$, the two exponentials for $|\xi| \to \infty$ and $\re \xi > 0$ on the physical sheet $\Xi_{+,+}$ become
									\begin{align*}
										&e^{2 \ii q_P(y-x)} = e^{-2(y-x) \xi} \left(1 + \tfrac{(y-x)  \omega^2}{(\hat{\lambda}_I + 2\hat{\mu}_I) \xi} + \tfrac{(y-x) ^2 \omega^4}{2(\hat{\lambda}_I + 2 \hat{\mu}_I)^2 \xi^2} + \mathcal{O}(|\xi|^{-3})  \right); \\
										& e^{\ii (y-x)(q_P + q_S)} = e^{-2(y-x) \xi}  \left( 1 +  \tfrac{(y-x) \rho \omega^2}{2 \hat{\mu}_I \xi} +  \tfrac{(y-x)^2 \rho^2 \omega^4}{8 \hat{\mu}^2_I \xi^2} + \mathcal{O}(|\xi|^{-3})  \right),
									\end{align*}
									because $q_P + q_S = 2 \ii \xi - \ii \frac{\omega^2 \rho}{2 \hat{\mu}_I \xi} + \mathcal{O}(\xi^{-3})$ and where $ \rho:= \frac{\hat{\lambda}_I + 3\hat{\mu}_I}{\hat{\lambda}_I + 2\hat{\mu}_I}$. Then expanding the terms of $\widetilde{\mathcal{G}}_2(x,y)$ we get
									\begin{align*}
										&- \frac{ \mathcal{A}(x) }{2 \ii q_P}  e^{2 \ii q_P(y-x)} = -e^{-2(y-x) \xi} \frac{ \mathcal{A}(x) }{- 2 \xi} \left(1 + \mathcal{O}(\xi^{-2})\right) \left(1 + \mathcal{O}(\xi^{-1})\right); \\
										& - \mathcal{B}(y) \frac{ e^{\ii (y-x)(q_P + q_S)}}{2 \ii q_S} = -e^{-2(y-x) \xi} \frac{ \mathcal{B}(y) }{- 2 \xi}  \left(1 + \mathcal{O}(\xi^{-2})\right) \left(1 + \mathcal{O}(\xi^{-1})\right); \\
										& \mathcal{C}  \frac{e^{\ii (y-x) (q_P + q_S)}}{2\omega^2}  = e^{-2(y-x) \xi} \frac{ \mathcal{C} }{2\omega^2}  \left(1  +  \frac{(y-x) \rho \omega^2}{2 \hat{\mu}_I \xi} + \mathcal{O}(|\xi|^{-2})   \right); \\
										&- \frac{\mathcal{C} }{2 \omega^2} e^{2 \ii q_P(y-x)} = - e^{-2(y-x) \xi} \frac{ \mathcal{C} }{2\omega^2} \left(1 + \frac{(y-x)  \omega^2}{(\hat{\lambda}_I + 2\hat{\mu}_I) \xi}  + \mathcal{O}(|\xi|^{-2})  \right).
									\end{align*}
									Summing up those terms we obtain
									\begin{align}\label{G 2 tilde obtained}
										\widetilde{ \mathcal{G}}_2(x,y) &= \frac{e^{-2(y-x) \xi}}{2 \xi} \left[ \mathcal{A}(x) + \mathcal{B}(y) + \frac{c_I (y-x)}{2 \hat{\mu}_I} \mathcal{C} + \mathcal{O}(\xi^{-1}) \right] \nonumber \\
										&= \frac{e^{-2(y-x) \xi}}{2 \xi} \left( 1 + \mathcal{O}(\xi^{-1})  \right)
									\end{align}
									using \eqref{A + B + C}. Plugging \eqref{G 2 tilde obtained} into $\int_{x}^{H} \widetilde{ \mathcal{G}}_2(x,y)V(y) \mathcal{H}_0(y,\xi) dy $ and recalling Lemma \ref{Lemma asymptotics of H0} we get
									\begin{align}\label{Integral of G2}
										&\int_{x}^{H} \widetilde{ \mathcal{G}}_2(x,y)V(y) \mathcal{H}_0(y,\xi) dy = \int_{x}^{H} \frac{e^{-2(y-x) \xi}}{2 \xi}  V(y) \left(-\xi  \frac{\hat{\mu}_I}{\omega^2} \right) \ma G_{11}^H & G_{11}^H  \\ \\ G_{12}^H  & G_{12}^H \am  = \mathcal{O}(|\xi|^{-\infty})
									\end{align}
									for $\re \xi >0$ by the dominated convergence theorem. The symbol $\mathcal{O}(|\xi|^{-\infty})$ means that the quantity is $\mathcal{O}(|\xi|^{-N})$ for any $N \in \mathbb{N}$.
								\end{itemize}
								Then \eqref{Integral of G1} and \eqref{Integral of G2} give the result.
							\end{proof}
							
							In the next lemma we compute the asymptotic expansion of the first Volterra iterate $\mathcal{H}^{(2)}(x,\xi)$ in \eqref{Expansion of Faddeev solution}.

							\begin{lemma}\label{Lemma asymptotics of second iterate of H}
								The second Volterra iterate $\mathcal{H}^{(2)}(x,\xi)$ in \eqref{Expansion of Faddeev solution} given by \eqref{Definition of Hl} for $|\xi| \to \infty$ and $\re \xi > 0$ in the physical sheet $\Xi_{+,+}$ admits the asymptotic expansion
								\begin{align*}
									\mathcal{H}^{(2)}(x,\xi) = \int_{x}^{H} \int_{y}^{H} \frac{e^{- \xi x}}{4\xi} \frac{\hat{\mu}_I}{\omega^2}V(y) V(t) G^H dy dt + O|\xi|^{-2}).
								\end{align*}
							\end{lemma}
							\begin{proof}
								From Lemma \ref{Lemma asymptotics of first iterate of H} we have that 
								\begin{align*}
									\widetilde{ \mathcal{G}}(y,t)V(t) \mathcal{H}_0(t,\xi) = \frac{\hat{\mu}_I}{2\omega^2} V(t) G^H + \mathcal{O}(\xi^{-1})
								\end{align*}
								while $\widetilde{ \mathcal{G}}(x,y) = -\frac{1}{2 \xi} + \mathcal{O}(|\xi|^{-2})$ by Lemma \ref{Lemma asymptotics of first iterate of H} for $\xi \in \Xi_{+,+}$ and $\re \xi >0$. So
								\begin{align*}
									\widetilde{ \mathcal{G}}(x,y) V(y) = - \frac{1}{2 \xi} V(y) + \mathcal{O}(\xi^{-2}),
								\end{align*}
								hence
								\begin{align*}
									&\int_{x}^{H} \int_{y}^{H} \widetilde{ \mathcal{G}}(x,y) V(y)  \widetilde{ \mathcal{G}}(y,t) V(t) \mathcal{H}_0(t,\xi) dy dt= - \int_{x}^{H} \int_{y}^{H} \frac{1}{4\xi} \frac{\hat{\mu}_I}{\omega^2}V(y) V(t) G^H dy dt + \mathcal{O}(|\xi|^{-2}).
									\qedhere
								\end{align*}
							\end{proof}
							
							In the next lemma we use the results of Lemma \ref{Lemma asymptotics of H0}, Lemma \ref{Lemma asymptotics of first iterate of H} and Lemma \ref{Lemma asymptotics of second iterate of H} in order to obtain the asymptotic expansion of the Jost solution in the physical sheet $\Xi_{+,+}$.

							\begin{lemma}\label{Lemma asymptotics on Jost solution}
								Let $V \in \mathcal{V}_{H}$, then the Jost solution $ \mathcal{F}(0,\xi) $ has the following asymptotic expansion for $|\xi| \to \infty$ and $\re \xi > 0$ in $\Xi_{+,+}$:
								\begin{align*}
									\mathcal{F}(0,\xi) &= -\xi  \frac{\hat{\mu}_I}{\omega^2} G^H + G_H(y)  - \frac{\hat{\mu}_I}{2\omega^2}  \int_{0}^{H} V(y) G^H dy  + \xi^{-1} \ma \frac{\hat{\mu}_I}{2(\hat{\lambda}_I + 2\hat{\mu}_I)} G_{11}^H  & 0 \\ \\ \frac{\hat{\mu}_I}{2(\hat{\lambda}_I + 2\hat{\mu}_I)} G_{12}^H  & 0 \am \\
									&\phantom{=\;}+ \frac{1}{4 \xi}   \int_{0}^{H}   c_I \mathcal{B}\left(\frac{y}{2}\right) y  V(y) G^H dy + \frac{1}{2 \xi} \int_{0}^{H} V(y) G_H(y) \;dy  \\
									&\phantom{=\;}- \frac{1}{4\xi}  \int_{0}^{H} \int_{y}^{H} \frac{\hat{\mu}_I}{\omega^2}V(y) V(t) G^H dy dt + \mathcal{O}(\xi^{-2}).
								\end{align*}
								
							\end{lemma}
							\begin{proof}
								From the three previous lemmas we have that
								
								\begin{align}\label{Expansion of H(x, xi)}
									& \mathcal{F}(x,\xi) e^{-\ii x q_P}=  \mathcal{H}(x,\xi) =-\xi  \frac{\hat{\mu}_I}{\omega^2} G^H + G_H(x) + \xi^{-1} \ma \frac{\hat{\mu}_I}{2(\hat{\lambda}_I + 2\hat{\mu}_I)} G_{11}^H  & - \frac{\omega^2 c^2_{I} x^2}{8 \hat{\mu}_I } G_{11}^H \\ \\ \frac{\hat{\mu}_I}{2(\hat{\lambda}_I + 2\hat{\mu}_I)} G_{12}^H  & - \frac{\omega^2 c^2_{I} x^2}{8 \hat{\mu}_I } G_{12}^H  \am \nonumber \\
									& - \int_{x}^{H} \frac{\hat{\mu}_I}{2\omega^2} V(y) G^H dy +\frac{1}{4 \xi} \frac{\hat{\mu}_I}{\omega^2}  \int_{x}^{H} \left(  \mathcal{B}(y) \frac{ \omega^2 c_I (y-x)}{\hat{\mu}_I} + \mathcal{C} \frac{\omega^2 c^2_{I} (y-x)^2}{4 \hat{\mu}^2_I} \right) V(y) G^H dy \nonumber \\
									& + \frac{1}{2 \xi} \int_{x}^{H} V(y) G_H(y) \; dy - \frac{1}{4\xi} \int_{x}^{H} \int_{y}^{H}  \frac{\hat{\mu}_I}{\omega^2}V(y) V(t) G^H dy dt + \mathcal{O}(\xi^{-2}).
								\end{align}
								Hence evaluating the Jost solution at $x=0$ we get
								\begin{align*}
									 \mathcal{F}(0,\xi) &= \mathcal{H}(0,\xi) = -\xi  \frac{\hat{\mu}_I}{\omega^2} G^H + G_H(y)  - \frac{\hat{\mu}_I}{2\omega^2}  \int_{0}^{H} V(y) G^H dy  \\
									&\phantom{=\;}  + \xi^{-1} \ma \frac{\hat{\mu}_I}{2(\hat{\lambda}_I + 2\hat{\mu}_I)} G_{11}^H  & 0 \\ \\ \frac{\hat{\mu}_I}{2(\hat{\lambda}_I + 2\hat{\mu}_I)} G_{12}^H  & 0 \am + \frac{1}{4 \xi}   \int_{0}^{H} \left(  c_I \mathcal{B}(y) y + \mathcal{C} \frac{ c^2_{I} y^2}{4 \hat{\mu}_I} \right) V(y) G^H dy\\
									&\phantom{=\;} + \frac{1}{2 \xi} \int_{0}^{H} V(y) G_H(y) \; dy  - \frac{1}{4\xi}  \int_{0}^{H} \int_{y}^{H} \frac{\hat{\mu}_I}{\omega^2}V(y) V(t) G^H dy dt + \mathcal{O}(\xi^{-2}),
								\end{align*}
								and using \eqref{Definition of A B and C}, we infer that
								\begin{align*}
									\mathcal{B}(y)  + \mathcal{C} \frac{ c_I y}{4 \hat{\mu}_I} =  \mathcal{B} \left(\frac{y}{2}\right).
								\end{align*}
								Thus, we can write
								\begin{align*}
									& \mathcal{F}(0,\xi) = -\xi  \frac{\hat{\mu}_I}{\omega^2} G^H + G_H(0)  - \frac{\hat{\mu}_I}{2\omega^2}  \int_{0}^{H} V(y) G^H dy    + \xi^{-1} \ma \frac{\hat{\mu}_I}{2(\hat{\lambda}_I + 2\hat{\mu}_I)} G_{11}^H  & 0 \\ \\ \frac{\hat{\mu}_I}{2(\hat{\lambda}_I + 2\hat{\mu}_I)} G_{12}^H  & 0 \am \\
									&+ \frac{1}{4 \xi}   \int_{0}^{H}   c_I \mathcal{B}\left(\frac{y}{2}\right) y  V(y) G^H dy  + \frac{1}{2 \xi} \int_{0}^{H} V(y) G_H(y) \; dy  - \frac{1}{4\xi}  \int_{0}^{H} \int_{y}^{H} \frac{\hat{\mu}_I}{\omega^2}V(y) V(t) G^H dy dt + \mathcal{O}(\xi^{-2}).
									\qedhere
								\end{align*}
							\end{proof}
							
							In the first proposition we use Lemma \ref{Lemma asymptotics on Jost solution} and \eqref{Expansion of H(x, xi)} to get an asymptotic expansion of the Jost function $	\mathcal{F}_\Theta (\xi)$ defined in \eqref{Jost function Markushevich}.
							
							\begin{proposition}\label{Proposition Jost func in phys sheet}
								For $V \in \mathcal{V}_{H}$, the Jost function in the physical sheet for $|\xi| \to \infty$ and $\re \xi > 0$ admits the asymptotic expansion
								\begin{equation*}
									\mathcal{F}_{\Theta}(\xi) = \xi^3 \chi_3 + \xi^2 \chi_2 + \xi \chi_1  + \chi_0 +  \ma \mathcal{O}(|\xi|^{-1}) &  \mathcal{O}(|\xi|^{-1}) \\ \mathcal{O}(1) & \mathcal{O}(1)  \am 
								\end{equation*}
								where 
								\begin{align*}
									\chi_3 = - \frac{\hat{\mu}_I}{\omega^2} \frac{2\hat{\mu}_I}{\hat{\mu}(0)} G_{11}^H \ma 0 & 0 \\ \\ 1 & 1 \am,
								\end{align*}
								\begin{align*}
									\chi_2 =& \frac{\hat{\mu}_I }{\omega^2} G^H  +  \frac{2\hat{\mu}_I}{\hat{\mu}(0)} \ma 0 & 0 \\ \\ G_{21}(0) & 0 \am -  \frac{\hat{\mu}^2_I}{\hat{\mu}(0)\omega^2} \int_{0}^{H} \left(V_{11}(y)G_{11}^H+V_{12}(y)G_{12}^H\right) \ma 0 & 0 \\ \\ 1 & 1 \am dy, 
								\end{align*}
								\begin{align*}
									\chi_1 &=  - G_H(0)   + \frac{\hat{\mu}_I}{2\omega^2}   \int_{0}^{H} V(y) G^H dy  + \frac{\hat{\mu}_I}{\omega^2} \ma \theta_3 G_{11}^H - \theta_2  G_{12}^H & \theta_3 G_{11}^H - \theta_2  G_{12}^H  \\ \\  \theta_1 G_{11}^H & \theta_1 G_{11}^H \am  \\
									&\phantom{=\;}+ \frac{\hat{\mu}^2_I}{\hat{\mu}(0)(\hat{\lambda}_I + 2\hat{\mu}_I )}G_{11}^H \ma 0 & 0 \\ \\ 1 & 0 \am  +\frac{\hat{\mu}_I}{\hat{\mu}(0)} \ma 0 & 0 \\ \\ 1 & 0 \am \int_{0}^{H} V(y) G_H(y) \; dy \\
									&\phantom{=\;}- \frac{\hat{\mu}^2_I}{2\hat{\mu}(0)\omega^2}  \ma 0 & 0 \\ \\ 1 & 0 \am \int_{0}^{H} \int_{y}^{H} V(y) V(t) G^H dy dt    +\frac{\hat{\mu}_I}{2\hat{\mu}(0)} \ma 0 & 0 \\ \\ 1 & 0 \am \int_{0}^{H} \left(  c_Iy \mathcal{B}(y/2)  \right) V(y) G^H   dy,
								\end{align*}
								and
								\begin{align*}
									\chi_0 =& - \frac{1}{2} G^H - \frac{\hat{\mu}_I}{2(\hat{\lambda}_I + 2\hat{\mu}_I) } \ma G_{11}^H & 0  \\ \\ G_{12}^H  & 0 \am - \frac{1}{2} \int_{0}^{H} V(y) G_H(y) \; dy  - \frac{1}{4 }   \int_{0}^{H} \left(  c_I y \mathcal{B}(y/2)  \right) V(y) G^H dy  \\
									&+ \frac{1}{4} \frac{\hat{\mu}_I}{\omega^2} \int_{0}^{H} \int_{y}^{H} V(y) V(t) G^H dy dt+  \ma -\theta_3 &\theta_2\\ \\ - \theta_1 & 0 \am    G_H(0)  - \frac{\hat{\mu}_I}{\omega^2} \frac{1}{2} \ma -\theta_3 &\theta_2\\ \\ - \theta_1 & 0 \am   \int_{0}^{H} V(y) G^H dy .
								\end{align*}

							\end{proposition}
							\begin{proof}
								Since
								\begin{equation*}
									\mathcal{F}(x,\xi) =  \mathcal{H}(x,\xi) e^{\ii x q_P},
								\end{equation*}
								we have
								\begin{equation*}
									\mathcal{F}'(x,\xi) = (\ii q_P) \mathcal{H}(x,\xi) e^{\ii x q_P} +  \mathcal{H}'(x,\xi) e^{\ii x q_P}
								\end{equation*}
								and
								\begin{align*}
									\mathcal{F}'(0,\xi) = (\ii q_P) \mathcal{H}(0,\xi)  +  \mathcal{H}'(0,\xi).
								\end{align*}
								We want to compute the derivative of the Jost solution up to the $\xi^0$ order.
								Differentiating \eqref{Expansion of H(x, xi)} we get that 
								\begin{align*}
									\mathcal{H}'_0(x,\xi) =  \ma G'_{21}(x) & -\frac{c_I }{2} G_{11}^H \\ \\  G_{22}'(x) & -\frac{c_I }{2} G_{12}^H \am  + \mathcal{O}(\xi^{-1})
								\end{align*}
								where $G_{21}^H(x) := -\frac{c_I}{2}G_{11}^H (x-H) + G_{21}^H$ and $G_{22}(x)= -\frac{c_I}{2}G_{12}^H (x-H) + G_{22}^H$, so
								\begin{align*}
									\mathcal{H}'_0(0,\xi) =  -\frac{c_I }{2} G^H   +\mathcal{O}(\xi^{-1}).
								\end{align*}
								We know that $\ii q_P = -\xi + \frac{{\omega^2}}{2(\hat{\lambda}_I + 2\hat{\mu}_I) \xi } + \mathcal{O}(|\xi|^{-3})$, so
								\begin{align*}
									(\ii q_P) \mathcal{H}(0,\xi) =& -\xi \mathcal{H}(0,\xi) + \frac{{\omega^2}}{2(\hat{\lambda}_I + 2\hat{\mu}_I) \xi } \mathcal{H}(0,\xi) + \mathcal{O}(\xi^{-2})=  \xi^2  \frac{\hat{\mu}_I}{\omega^2} G^H - \xi G_H(0)  + \xi \frac{\hat{\mu}_I}{2\omega^2}  \int_{0}^{H} V(y) G^H dy  \\
									&  -\ma \frac{\hat{\mu}_I}{2(\hat{\lambda}_I + 2\hat{\mu}_I)} G_{11}^H  & 0 \\ \\ \frac{\hat{\mu}_I}{2(\hat{\lambda}_I + 2\hat{\mu}_I)} G_{12}^H  & 0 \am - \frac{1}{4 }   \int_{0}^{H} \left(  c_Iy \mathcal{B}(y/2)   \right) V(y) G^H dy  -\frac{1}{2 } \int_{0}^{H} V(y) G_H(y) dy \\
									& + \frac{1}{4}  \int_{0}^{H} \int_{y}^{H} \frac{\hat{\mu}_I}{\omega^2}V(y) V(t) G^H dy dt + \frac{{\omega^2}}{2(\hat{\lambda}_I + 2\hat{\mu}_I) \xi } (-\xi)  \frac{\hat{\mu}_I}{\omega^2} G^H  + \mathcal{O}(\xi^{-2}).
								\end{align*}
								Finally, we have 
								\begin{align*}
									\mathcal{F}'(0,\xi) &=  \xi^2  \frac{\hat{\mu}_I}{\omega^2} G^H - \xi G_H(0)  + \xi \frac{\hat{\mu}_I}{2\omega^2}  \int_{0}^{H} V(y) G^H dy     - \frac{1}{4 }   \int_{0}^{H} \left(  c_Iy \mathcal{B}(y/2)   \right) V(y) G^H dy \\
									&\phantom{=\;} -\frac{1}{2 } \int_{0}^{H} V(y) G_H(y) dy + \frac{1}{4}  \int_{0}^{H} \int_{y}^{H} \frac{\hat{\mu}_I}{\omega^2}V(y) V(t) G^H dy dt  -\ma \frac{\hat{\mu}_I}{2(\hat{\lambda}_I + 2\hat{\mu}_I)} G_{11}^H  & 0 \\ \\ \frac{\hat{\mu}_I}{2(\hat{\lambda}_I + 2\hat{\mu}_I)} G_{12}^H  & 0 \am \\
									&\phantom{=\;}- \frac{ \hat{\mu}_I }{2(\hat{\lambda}_I + 2\hat{\mu}_I) }   G^H -\frac{c_I }{2} G^H + \frac{\hat{\mu}_I}{\omega^2} V(x) G^H + \mathcal{O}(\xi^{-2}).
								\end{align*}
								Adding together the last two matrices we get
								\begin{align*}
									&- \frac{ \hat{\mu}_I}{2(\hat{\lambda}_I + 2\hat{\mu}_I) }  G^H  -  \frac{\hat{\lambda}_I + \hat{\mu}_I}{2(\hat{\lambda}_I + 2\hat{\mu}_I)} G^H  = - \frac{1}{2} G^H 
								\end{align*}
								and thus
								\begin{align}\label{derivative of Jost sol at 0}
									&\mathcal{F}'(0,\xi) =  \xi^2  \frac{\hat{\mu}_I}{\omega^2} G^H- \xi G_H(0)  + \xi \frac{\hat{\mu}_I}{2\omega^2}  \int_{0}^{H} V(y) G^H dy   - \frac{1}{2} G^H - \frac{\hat{\mu}_I}{2(\hat{\lambda}_I + 2\hat{\mu}_I) } \ma G_{11}^H & 0  \\ \\ G_{12}^H  & 0 \am  \nonumber \\
									&- \frac{1}{4 }   \int_{0}^{H} \left(  c_Iy \mathcal{B}(y/2)   \right) V(y) G^H dy   -\frac{1}{2 } \int_{0}^{H} V(y) G_H(y) \; dy  + \frac{1}{4}  \int_{0}^{H} \int_{y}^{H} \frac{\hat{\mu}_I}{\omega^2}V(y) V(t) G^H dy dt  + \mathcal{O}(\xi^{-2}).
								\end{align}
								Now we are able to calculate the expansion of $\mathcal{F}_{\Theta}(\xi) = \mathcal{F}'(0,\xi) + \Theta \mathcal{F}(0,\xi)$, but we should reflect on the fact that we calculated $\mathcal{F}(0,\xi)$ up to $\xi^{-1}$ order and $\mathcal{F}'(0,\xi)$ up to $\xi^0$ order. Since
								\begin{equation*}
									\Theta :=\ma-\theta_3 &\theta_2\\ \\
									2\frac{\hat{\mu}_I}{\hat{\mu}(0)}\xi^2-\theta_1&0\am = \xi^2 \frac{2\hat{\mu}_I}{\hat{\mu}(0)} \ma 0 & 0 \\ \\ 1 & 0 \am + \ma -\theta_3 &\theta_2\\ \\ - \theta_1 & 0 \am,
								\end{equation*}
								$\mathcal{F}_{\Theta}(\xi) $ assumes the form
								\begin{equation*}
									\mathcal{F}_{\Theta}(\xi) = \ma \xi^2 \chi_{2,11} + \xi \chi_{1,11} + \chi_{0,11} +\mathcal{O}(\xi^{-1}) & \xi^2\chi_{2,12}+ \xi \chi_{1,12}+ \chi_{0,12}+ \mathcal{O}(\xi^{-1}) \\ \\
									\xi^3\chi_{3,21}+ \xi^2 \chi_{2,21}+ \xi \chi_{1,21}+ \mathcal{O}(1) & \xi^3 \chi_{3,22}+ \xi^2 \chi_{2,22}+ \xi\chi_{1,22}+ \mathcal{O}(1) \am
								\end{equation*}
								where $\chi_{i,jk}$ is the $jk$ component of the matrix $\chi_i$.
								Then, to know the terms of order $\xi^0$ in the second row, we should have calculated the Jost solution $\mathcal{F}(0,\xi)$ up to the order $\xi^{-2}$.  The matrix $\Theta$ can be decomposed into $\Theta= \xi^2 M^1 + M^2$. Multiplying $M^1$ by any matrix leads to a matrix whose only non-zero terms are in the second row. Hence, we multiply $\xi^2 M^1$  by terms of $\mathcal{F}(0,\xi)$ of order up to $\xi^{-1}$, whereas we multiply $M^2$ by terms of $\mathcal{F}(0,\xi)$ of order up to $\xi^{0}$.
								Then we have
								\begin{align}\label{theta times Jost solution}
									&\Theta \mathcal{F}(0,\xi) = - \xi^3  \frac{\hat{\mu}_I}{\omega^2} \frac{2\hat{\mu}_I}{\hat{\mu}(0)} \ma 0 & 0 \\ \\ 1 & 0 \am G^H + \xi^2 \frac{2\hat{\mu}_I}{\hat{\mu}(0)} \ma 0 & 0 \\ \\ 1 & 0 \am G_H(0) - \xi^2  \frac{\hat{\mu}_I}{2\omega^2} \frac{2\hat{\mu}_I}{\hat{\mu}(0)} \ma 0 & 0 \\ \\ 1 & 0 \am \int_{0}^{H} V(y) G^H dy \nonumber \\
									& - \xi \frac{\hat{\mu}_I}{\omega^2}  \ma -\theta_3 &\theta_2\\ \\ - \theta_1 & 0 \am G^H + \xi  \frac{\hat{\mu}_I}{\hat{\mu}(0)} \frac{\hat{\mu}_I}{\hat{\lambda}_I + 2\hat{\mu}_I} \ma 0 & 0 \\ \\ 1 & 0 \am \ma G_{11}^H & 0  \\ \\ G_{12}^H  & 0 \am  + \xi \frac{\hat{\mu}_I}{\hat{\mu}(0)} \ma 0 & 0 \\ \\ 1 & 0 \am \int_{0}^{H} V(y) G_H(y) dy \nonumber \\
									&-\xi \frac{\hat{\mu}_I}{2\hat{\mu}(0)} \frac{\hat{\mu}_I}{\omega^2} \ma 0 & 0 \\ \\ 1 & 0 \am \int_{0}^{H} \int_{y}^{H} V(y) V(t) G^H dy dt +\xi \frac{\hat{\mu}_I}{2\hat{\mu}(0)} \ma 0 & 0 \\ \\ 1 & 0 \am \int_{0}^{H} \left(  c_I y \mathcal{B}(y/2)  \right) V(y) G^H   dy \nonumber \\
									& + \ma -\theta_3 &\theta_2\\ \\ - \theta_1 & 0 \am     G_H(0) - \frac{\hat{\mu}_I}{\omega^2} \frac{1}{2} \ma -\theta_3 &\theta_2\\ \\ - \theta_1 & 0 \am   \int_{0}^{H} V(y) G^H dy + \ma \mathcal{O}(1) & \mathcal{O}(1) \\ \mathcal{O}(|\xi|) & \mathcal{O}(|\xi|)  \am .
								\end{align}
								Adding \eqref{theta times Jost solution} and \eqref{derivative of Jost sol at 0} yields the statement of the proposition.
							\end{proof}
							
							A similar result to Proposition \ref{Proposition Jost func in phys sheet} can be found in \cite{Iantchenko} but with only two orders of expansion and without proof. 
							
							In the following, we will compute the asymptotic expansion of the determinant of the Jost function for $\xi$ in the physical sheet $\Xi_{+,+}$ (Lemma \ref{lemma det on phys sheet}) and then we do the same for the sheet $\Xi_{-,-}$ (Lemma \ref{lemma det on unphys sheet}) and the other sheets as in (Lemma \ref{Lemma det P(xi)}. The goal is to find an asymptotic expansion for the entire function $F(\xi)$ as the product of all the Rayleigh determinants (see Theorem \ref{Lemma asymptotic of F}) and an exponential type estimate for $F(\xi)$ (see Theorem \ref{F is of expponential type}).
							\begin{lemma}\label{lemma det on phys sheet}
								Let $V \in \mathcal{V}_{H}$, then the determinant of the Jost function for $|\xi| \to \infty$ and $\re \xi > 0$ on the physical sheet $\Xi_{+,+}$ satisfies
								\begin{equation*}
									\det \mathcal{F}_{\Theta}(\xi) = \xi^3  \frac{\hat{\mu}_I}{\omega^2} c(0) + \mathcal{O}(|\xi|^{2})
								\end{equation*}
								where $c(0) := \frac{\hat{\lambda}(0) + \hat{\mu}(0)}{\hat{\lambda}(0) + 2\hat{\mu}(0)}$.
							\end{lemma}
							\begin{proof}
								We can write $\det \mathcal{F}_{\Theta}(\xi) = a \xi^5 + b \xi^4 + c \xi^3 + \mathcal{O}(\xi^2)$ and we can see that
								\begin{align*}
									a= \frac{\hat{\mu}_I}{\omega^2} G_{11}^H \left(- \frac{\hat{\mu}_I}{\omega^2} \frac{2\hat{\mu}_I}{\hat{\mu}(0)} G_{11}^H \right) + \frac{\hat{\mu}_I}{\omega^2} G_{11}^H \left( \frac{\hat{\mu}_I}{\omega^2} \frac{2\hat{\mu}_I}{\hat{\mu}(0)} G_{11}^H \right) = 0.
								\end{align*}
								The coefficient $b$ is obtained by multiplication of the first row of the terms of order $\xi^2$ with the second row of the terms of order $\xi^2$ and by the second row of the term of order  $\xi^3$ multiplied by the first row of the term of order $\xi$. Hence, only the first three terms play a role in the computation of $b$, as the other three terms have zero $11-$ and $12-$components. Thus,
								\begin{align*}
									&b= - \frac{\hat{\mu}_I}{\omega^2} G_{11}^H \frac{2\hat{\mu}_I}{\hat{\mu}(0)} G_{21}^H(0)  - \frac{\hat{\mu}_I}{\omega^2} \frac{2\hat{\mu}_I}{\hat{\mu}(0)} G_{11}^H \left[ - G_{21}^H(0) \right] = 0.
								\end{align*} 
								When calculating the determinant, often one factor of a product of matrices is a $2\times2$-matrix having two identical columns, then the resulting product has zero determinant.
								Keeping this in mind, we can calculate the coefficient $c$. First, we calculate the contribution to $c$ given by the elements in the first row of order $\xi^0$  multiplied by the elements of the second row of order $\xi^3$
								\begin{align*}
									&c^{(03)} = \frac{2 \hat{\mu}_I^2}{\omega^2 \hat{\mu}(0)} \frac{\hat{\mu}_I}{2(\hat{\lambda}_I + 2\hat{\mu}_I)} (G_{11}^H)^2  +\frac{2 \hat{\mu}_I^2}{\omega^2 \hat{\mu}(0)} G_{11}^H \left[ \theta_3 G_{21}^H(0) - \theta_2 G_{22}^H(0) \right] \\
									&+  \frac{ \hat{\mu}_I^2}{\omega^2 \hat{\mu}(0)} G_{11}^H  \int_{0}^{H}  \left[ V_{11}(y) G_{21}^H(y) + V_{12}(y)  G_{22}^H(y)\right]dy + \frac{ \hat{\mu}_I^2}{\omega^2 \hat{\mu}(0)} G_{11}^H  \int_{0}^{H} \frac{c_I y}{2}  \left[ V_{11}(y) G_{11}^H + V_{12}(y)  G_{12}^H \right]dy 
								\end{align*}
								where $G_{21}^H(y) := -\frac{c_I}{2}G_{11}^H (y-H) + G_{21}^H$ and $G_{22}(y)= -\frac{c_I}{2}G_{12}^H (y-H) + G_{22}^H$. Then we calculate the determinant considering the elements in the first row of order $\xi^2$  multiplied by the elements of the second row of order $\xi^1$
								\begin{align*}
									c^{(21)} &= \frac{\hat{\mu}_I}{\omega^2} G_{11}^H G_{22}^H(0) - \frac{ \hat{\mu}_I}{\omega^2 } \frac{\hat{\mu}^2_I}{\hat{\mu}(0)(\hat{\lambda}_I + 2\hat{\mu}_I)} (G_{11}^H)^2 + \frac{ \hat{\mu}_I^2}{\omega^2 \hat{\mu}(0)} G_{11}^H  \int_{0}^{H} -\frac{c_I y}{2}  \left[ V_{11}(y) G_{11}^H + V_{12}(y)  G_{12}^H \right]dy \\
									&\phantom{=\;} - \frac{ \hat{\mu}_I^2}{\omega^2 \hat{\mu}(0)} G_{11}^H  \int_{0}^{H}  \left[ V_{11}(y) G_{21}^H(y) + V_{12}(y)  G_{22}^H(y)\right]dy.
								\end{align*}
								Finally, we calculate the determinant considering the elements in the first row of order $\xi^1$  multiplied by the elements of the second row of order $\xi^2$:
								\begin{align*}
									c^{(12)} = (- G_{21}^H(0)) \frac{\hat{\mu}_I}{\omega^2} G_{12}^H  -  \left[ \theta_3 G_{11}^H - \theta_2 G_{12}^H \right]  \frac{ \hat{\mu}_I}{\omega^2}  \frac{2 \hat{\mu}_I}{\hat{\mu}(0)} G_{21}^H(0).
								\end{align*}
								Hence, summing $c^{(03)} + c^{(21)} + c^{(12)}$, we get $c$
								\begin{align*}
									c = \frac{\hat{\mu}_I}{\omega^2} - \frac{2 \theta_2 \hat{\mu}_I^2}{\omega^2 \hat{\mu}(0)}.
								\end{align*}
								Now we recall that $\theta_2 := \frac{1}{2\hat{\mu}_I} \frac{\hat{\mu}(0)}{\hat{\lambda}(0) + 2\hat{\mu}(0)}$,  so $c= \frac{\hat{\mu}_I}{\omega^2}(1 -  \frac{\hat{\mu}(0)}{\hat{\lambda}(0) + 2\hat{\mu}(0)})$. Then, the determinant of the Jost function $\det \mathcal{F}_{\Theta}(\xi)$ for $\xi$ in the physical sheet and $|\xi| \to \infty $, $\re \xi > 0$ satisfies
								\begin{align*}
									\det \mathcal{F}_{\Theta}(\xi) &= \xi^3 \left( \frac{\hat{\mu}_I}{\omega^2} \left[ 1 -  \frac{\hat{\mu}(0)}{\hat{\lambda}(0) + 2\hat{\mu}(0)}\right] \right) + \mathcal{O}(|\xi|^{2})= \xi^3 \left( \frac{\hat{\mu}_I}{\omega^2} c(0) \right) + \mathcal{O}(|\xi|^{2}).
									\qedhere
								\end{align*}
							\end{proof}
							We define
							\begin{align*}
								&\widetilde{G}^H = \ma G_{11}^H & - G_{11}^H  \\ \\ G_{12}^H  & - G_{12}^H \am, \qq \qq \widetilde{G}_H(y) = \ma  G_{21}(y)   &  \frac{c_I y}{2}G_{11}^H \\ \\ G_{22}(y)  &  \frac{c_I y}{2}G_{12}^H  \am,
							\end{align*}
							where, $\widetilde{G}^H $ is obtained from $G^H$ by inverting the sign in the second column, while $\widetilde{G}_H(y)$ is obtained from $G_H(y)$ by inverting the sign in the second column.
							
							In the following lemma we compute the asymptotics of $\det \mathcal{F}_{\Theta}(w_{PS}(\xi)) $ in the sheet $\Xi_{+,+}$ of the Riemann surface $\Xi$, that is equal to $\det \mathcal{F}_{\Theta} (\xi)$ in the sheet $\Xi_{-,-}$.

							\begin{lemma}\label{lemma det on unphys sheet}
								Let $V \in \mathcal{V}_{H}$, then the determinant of the Jost function $\mathcal{F}_{\Theta}(w_{PS}(\xi))$ for $\re \xi > 0$ and as $\re \xi \to \infty$ in $\Xi_{+,+}$ is
								\begin{equation*}
									\det \mathcal{F}_{\Theta}(w_{PS}(\xi)) = \xi^3  \frac{\hat{\mu}_I}{\omega^2} c(0) + \xi^3 \mathscr{A}(\xi) + \xi^2\mathscr{B}(\xi) + \mathscr{R}(\xi),
								\end{equation*}
								where 
								\begin{align*}
									\mathscr{A}(\xi) := \frac{2\hat{\mu}_I^2}{\hat{\mu}(0) \omega^2} \int_{0}^{H} e^{2\xi y} V_{12}(y) dy
								\end{align*}
								and $\mathscr{B}(\xi)$ can be written as 
								\begin{align*}
									\mathscr{B}(\xi) = C_1 \left( \int_{0}^{H} f_1(y) e^{2y\xi} dy \right) \left( \int_{0}^{H} f_2(t) e^{2t\xi} dt\right)
								\end{align*}
								with $f_1$ and $f_2$ being in $L^1$. The term $\mathscr{R}(\xi)$ is a remainder term containing all the terms polynomially smaller than $\xi^3 \mathscr{A}(\xi)$ and all the other terms dominated by $\xi^2 \mathscr{B}(\xi)$. 
							\end{lemma}
							For simplicity $\mathscr{R}(\xi)$ denotes a remainder as in the statement that will be allowed to change between occurrences.
							\begin{proof}
								It is clear that $\mathcal{F}_{\Theta}(w_{PS}(\xi)) $ for $\xi \in \Xi_{+,+}$ is equal to  $\mathcal{F}_{\Theta}(\xi) $ for $\xi \in \Xi_{-,-}$.
								When we consider the Jost solution in the unphysical sheet $\Xi_{-,-}$, what changes is the expansion of $q_P$ and $q_S$, which, for $\re \xi \to +\infty$ and $\xi \in \Xi_{-,-}$, is
								\begin{align*}
									&\ii q_P =  \xi - \frac{\omega^2}{2(\hat{\lambda}_I + 2\hat{\mu}_I) \xi } + \mathcal{O}(|\xi|^{-3});
									&\ii q_S =  \xi - \frac{\omega^2}{2\hat{\mu}_I \xi } + \mathcal{O}(|\xi|^{-3}); \\
									&e^{-\ii x (q_P-q_S)} = \left(1 - \frac{xc_I \omega^2x}{ 2\hat{\mu}_I \xi} + \frac{c_I^2  \omega^4 x^2}{8 \hat{\mu}_I^2 \xi^2} + \mathcal{O}(|\xi|^{-3})  \right).
								\end{align*}
								From the expansions above and the definition of $\mathcal{H}_0(x,\xi)$, it is clear that the unperturbed Faddeev solution in the sheet $\Xi_{-,-}$ is the same as in the physical sheet but with the first column with an opposite sign for odd powers of $\xi$ and with the second column with an opposite sign for even powers of $\xi$ as below
								\begin{align*}
									\mathcal{H}_0(x,\xi) =& \xi  \frac{\hat{\mu}_I}{\omega^2} \ma G_{11}^H & -G_{11}^H  \\ \\ G_{12}^H  & -G_{12}^H \am + \ma G_{21}(x) & \frac{c_I x}{2} G_{11}^H \\ \\  G_{22}(x) & \frac{c_I x}{2} G_{12}^H \am   +\xi^{-1} \ma -\frac{\hat{\mu}_I}{2(\hat{\lambda}_I + 2\hat{\mu}_I)} G_{11}^H  & - \frac{\omega^2 c^2_{I} x^2}{8 \hat{\mu}_I } G_{11}^H  \\ \\ -\frac{\hat{\mu}_I}{2(\hat{\lambda}_I + 2\hat{\mu}_I)} G_{12}^H  & - \frac{\omega^2 c^2_{I} x^2}{8 \hat{\mu}_I } G_{12}^H  \am + \mathcal{O}(|\xi|^{-2}).
								\end{align*}
								For the Green function $\widetilde{ \mathcal{G}}_1(x,y)$ instead there is invariance under change of sign on $q_S$ and $q_P$ (even function with respect to $q_P$ and $q_S$). Then, $\int_{x}^{H} 	\widetilde{ \mathcal{G}}_1(x,y) V(y)\mathcal{H}_0(y,\xi) dy$ has the same property as $\mathcal{F}_0(0,\xi)$ and assumes the form
								\begin{align*}
									&\int_{x}^{H} 	\widetilde{ \mathcal{G}}_1(x,y) V(y)\mathcal{H}_0(y,\xi) dy = \int_{x}^{H} \frac{\hat{\mu}_I}{2\omega^2} V(y) \widetilde{G}^H dy \nonumber \\
									& - \frac{1}{4 \xi} \frac{\hat{\mu}_I}{\omega^2}  \int_{x}^{H} \left(  \mathcal{B}(y) \frac{ \omega^2 c_I (y-x)}{\hat{\mu}_I} + \mathcal{C} \frac{\omega^4 c^2_{I} (y-x)^2}{4 \hat{\mu}^2_I} \right) V(y) \widetilde{G}^H dy  - \frac{1}{2 \xi} \int_{x}^{H} V(y) \widetilde{G}_H(y) dy + O\left(\frac{1}{|\xi|^2}\right).
								\end{align*}
								We know also that in $\Xi_{-,-}$ the following asymptotic expansions hold
								\begin{align*}
									&e^{2 \ii q_P(y-x)} = e^{2(y-x) \xi} \left(1 - \frac{(y-x)  \omega^2}{(\hat{\lambda}_I + 2\hat{\mu}_I) \xi} + \frac{(y-x) ^2 \omega^4}{2(\hat{\lambda}_I + 2 \hat{\mu}_I)^2 \xi^2} + \mathcal{O}(|\xi|^{-3})  \right); \\
									& e^{\ii (y-x)(q_P + q_S)} = e^{2(y-x) \xi}  \left( 1 - \frac{(y-x) \rho \omega^2}{2 \hat{\mu}_I \xi} +  \frac{(y-x)^2 \rho^2 \omega^4}{8 \hat{\mu}^2_I \xi^2} + \mathcal{O}(|\xi|^{-3})  \right),
								\end{align*}
								because $q_P + q_S = 2 \ii \xi - \ii \frac{\omega^2 \rho}{2 \hat{\mu}_I \xi} + \mathcal{O}(\xi^{-3})$ and where $ \rho:= \frac{\hat{\lambda}_I + 3\hat{\mu}_I}{\hat{\lambda}_I + 2\hat{\mu}_I}$. Then $\widetilde{ \mathcal{G}}_2(x,y)$ has the same asymptotic expansion as in the physical sheet after replacing $\xi$ with $-\xi$, namely
								\begin{align*}
									\widetilde{ \mathcal{G}}_2(x,y) = - \frac{e^{2 \xi (y-x)}}{2 \xi} + \frac{e^{2 \xi (y-x)}}{2 \xi^2} \frac{(y-x)\omega^2}{2\hat{\hat{\mu}}_I \sigma_I} \left( \mathfrak{d}(x,y) + \mathcal{O}(|\xi|^{-1}\right)
								\end{align*}
								where 
								\begin{align*}
									\mathfrak{d}(x,y) := (\hat{\lambda} + 5\hat{\mu}) \left(\mathcal{A}(x) +  \mathcal{B}(y)\right)+ \frac{(y-x)}{2\hat{\mu} \sigma} \mathcal{C} \left(\hat{\lambda}^2 + 5\hat{\mu}^2 + 6 \hat{\lambda} \hat{\mu}\right).
								\end{align*}
								Then, the integral term $\int_{x}^{H} 	\widetilde{ \mathcal{G}}_2(x,y) V(y)\mathcal{H}_0(y,\xi) dy$ becomes
								\begin{align*}
									 \int_{x}^{H} 	\widetilde{ \mathcal{G}}_2(x,y) V(y)\mathcal{H}_0(y,\xi) dy &= -  \frac{\hat{\mu}_I}{2\omega^2} \int_{x}^{H} e^{2(y-x) \xi}  V(y) \widetilde{G}^H dy- \frac{1}{2\xi} \int_{x}^{H} e^{2(y-x) \xi}  V(y) \widetilde{G}_H(y)  dy \\
									&\phantom{=\;}+ \frac{1}{2\xi} \int_{x}^{H} e^{2(y-x) \xi} \frac{y-x}{2 \sigma} \mathfrak{d}(x,y)  V(y) \widetilde{G}^H \left(1 + \mathcal{O}(|\xi|^{-1})\right)dy.
								\end{align*}
								Adding all the terms, the Faddeev solution becomes
								\begin{align*}
									\mathcal{H}(x,\xi) &= \xi  \frac{\hat{\mu}_I}{\omega^2} \widetilde{ G}^H + \widetilde{ G}_H(x) + \xi^{-1} \ma -\frac{\hat{\mu}_I}{2(\hat{\lambda}_I + 2\hat{\mu}_I)} G_{11}^H  & - \frac{\omega^2 c^2_{I} x^2}{8 \hat{\mu}_I } G_{11}^H \\ \\ -\frac{\hat{\mu}_I}{2(\hat{\lambda}_I + 2\hat{\mu}_I)} G_{12}^H  & - \frac{\omega^2 c^2_{I} x^2}{8 \hat{\mu}_I } G_{12}^H  \am  - \int_{x}^{H} \frac{\hat{\mu}_I}{2\omega^2} V(y) \widetilde{ G}^H dy \nonumber \\
									&\phantom{=\;}-\frac{1}{4 \xi}  \int_{x}^{H} c_I (y-x) \left(  \mathcal{B}(y)  + \mathcal{C} \frac{ c_{I} (y-x)}{4 \hat{\mu}_I} \right) V(y) \widetilde{ G}^H dy  - \frac{1}{2 \xi} \int_{x}^{H} V(y) \widetilde{ G}_H(y) \; dy + \mathcal{O}(|\xi|^{-2}) \\
									&\phantom{=\;}+  \frac{\hat{\mu}_I}{2\omega^2} \int_{x}^{H} e^{2(y-x) \xi}  V(y) \widetilde{G}^H dy +\frac{1}{2\xi} \int_{x}^{H} e^{2(y-x) \xi}  V(y)  \widetilde{G}_H(y) dy \\
									&\phantom{=\;} - \frac{1}{2\xi} \int_{x}^{H} e^{2(y-x) \xi} \frac{y-x}{2 \sigma} \mathfrak{d}(x,y)  V(y) \widetilde{G}^H \left(1 + \mathcal{O}(|\xi|^{-1})\right)dy
								\end{align*}
								which is the same as in the physical sheet inverting the first column of the terms with odd powers of $\xi$ ($G^H$ and $G_H(y)$ to $-\widetilde{ G}^H$ and $-\widetilde{ G}_H(y)$) and the second column of the terms with even powers of $\xi$ ($G^H$ and $G_H(y)$ to $\widetilde{ G}^H$ and $\widetilde{ G}_H(y)$), plus the last two terms, which are exponentially large and that, instead, were vanishing in the physical sheet. Let 
								\begin{align*}
									&\mathcal{D}_1(x,\xi) :=  \frac{\hat{\mu}_I}{2\omega^2} \int_{x}^{H} e^{2(y-x) \xi}  V(y) \widetilde{G}^H dy \\
									&\mathcal{D}_2(x,\xi) := \frac{1}{2\xi} \int_{x}^{H} e^{2(y-x) \xi}  V(y)  \widetilde{G}_H(y)  dy \\
									&\mathcal{D}_3(x,\xi) := - \frac{1}{2\xi} \int_{x}^{H} e^{2(y-x) \xi} \frac{y-x}{2 \sigma} \mathfrak{d}(x,y)  V(y) \widetilde{G}^H,
								\end{align*} 
								then the Jost function $\mathcal{F}_{\Theta}(w_{PS}(\xi))$ is the same as in the physical sheet, after the sign replacement explained above, plus all the contributions coming from the exponentially large terms $\mathcal{D}_1(x,\xi) $, $\mathcal{D}_2(x,\xi) $ and $\mathcal{D}_3(x,\xi)$. These $\mathcal{D}_i(x,\xi)$ terms are exponentially large since $V$ is continuous and non-zero in the set $(H-\epsilon,H)$ and thus it has a definite sign in $(H-\epsilon,H)$. Then the part of the integral close to $H$ has no cancellation effects and dominates the rest.
								For the Jost function, we need to calculate
								\begin{align*}
									\sum_{i=1}^{3}(\ii q_P) \mathcal{D}_i(0,\xi) + \mathcal{D}'_i(0,\xi) + \Theta \mathcal{D}_i(0,\xi).  
								\end{align*}
								On the one hand
								\begin{align*}
									\sum_{i=1}^{3}(\ii q_P) \mathcal{D}_i(0,\xi) + \mathcal{D}'_i(0,\xi) =& -\xi  \frac{\hat{\mu}_I}{2\omega^2} \int_{0}^{H} e^{2y \xi}  V(y) \widetilde{G}^H   \\
									&- \frac{1}{2 } \int_{0}^{H} e^{2y \xi}  V(y) \widetilde{ G}_H(y) dy  + \frac{1}{2} \int_{0}^{H} e^{2 y \xi} \frac{y}{2 \sigma} \mathfrak{d}(0,y)  V(y) \widetilde{G}^H 
								\end{align*}
								while on the other hand
								\begin{align*}
									\sum_{i=1}^{3} \Theta \mathcal{D}_i(0,\xi) &= \xi^2 \frac{2 \hat{\mu}_I}{\hat{\mu}(0)}  \frac{\hat{\mu}_I}{2\omega^2} \ma 0 & 0 \\ 1 & 0 \am \int_{0}^{H} e^{2y \xi}  V(y) \widetilde{G}^H dy+ \xi \frac{\hat{\mu}_I}{\hat{\mu}(0) } \ma 0 & 0 \\ 1 & 0 \am  \int_{0}^{H} e^{2y \xi}  V(y) \widetilde{ G}_H(y) dy \\
									&\phantom{=\;}- \frac{\xi}{2}  \frac{2 \hat{\mu}_I}{\hat{\mu}(0)} \ma 0 & 0 \\ 1 & 0 \am  \int_{0}^{H} e^{2y\xi} \frac{y}{2 \sigma} \mathfrak{d}(0,y)  V(y) \widetilde{G}^H  + \frac{\hat{\mu}_I}{2\omega^2} \ma -\theta_3 & \theta_2 \\ -\theta_1 & 0 \am \int_{0}^{H} e^{2 y \xi}  V(y) \widetilde{G}^H dy \\
									&\phantom{=\;}+ \frac{1}{2 \xi}  \ma -\theta_3 & \theta_2 \\ -\theta_1 & 0 \am \int_{0}^{H} e^{2y \xi} V(y) \widetilde{ G}_H(y) dy - \frac{1}{2 \xi} \ma -\theta_3 & \theta_2 \\ -\theta_1 & 0 \am \int_{0}^{H} e^{2 y \xi} \frac{y}{2 \sigma} \mathfrak{d}(0,y)  V(y) \widetilde{G}^H.
								\end{align*}
								Then $\mathcal{F}_{\Theta}(w_{PS}(\xi))$ admits the form
								\begin{equation*}
									\mathcal{F}_{\Theta}(w_{PS}(\xi))= \xi^3 \chi_3^{PS} + \xi^2 \chi_2^{PS}  + \xi \chi_1^{PS}   + \chi_0^{PS}  + E(\xi)+ \mathscr{R}(\xi)
								\end{equation*}
								with $\chi^{PS}_i$, for $i=0, ... \,,3$, being the same as in the physical sheet after inverting the sign, according to the rules mentioned above. The term $E(\xi)$ contains the exponentially large terms and the remainder $\mathscr{R}(\xi)$ contains the polynomially lower order terms and the polynomially lower order terms of the terms containing the exponential $e^{2 \xi y}$. In particular
								\begin{align*}
									\chi_3^{PS} = \frac{\hat{\mu}_I}{\omega^2} \frac{2\hat{\mu}_I}{\hat{\mu}(0)} G_{11}^H \ma 0 & 0 \\ \\ 1 & -1 \am,
								\end{align*}
								\begin{align*}
									\chi_2^{PS} =& \frac{\hat{\mu}_I}{\omega^2} \widetilde{G}^H +  \frac{2\hat{\mu}_I}{\hat{\mu}(0)} G_{21}(0) \ma 0 & 0 \\ \\ 1 & 0 \am  -  \frac{\hat{\mu}^2_I}{\hat{\mu}(0)\omega^2} \int_{0}^{H} \left(V_{11}(y)G_{11}^H+V_{12}(y)G_{12}^H\right) \ma 0 & 0 \\ \\ 1 & -1 \am dy ,
								\end{align*}
								\begin{align*}
									\chi_1^{PS}  &= \widetilde{G}_H(0)   - \frac{\hat{\mu}_I}{2\omega^2}   \int_{0}^{H} V(y) \widetilde{G}^H  dy  - \frac{\hat{\mu}_I}{\omega^2} \ma \theta_3 G_{11}^H - \theta_2  G_{12}^H &  \theta_2  G_{12}^H -\theta_3 G_{11}^H  \\ \\  \theta_1 G_{11}^H & -\theta_1 G_{11}^H \am  - \frac{\hat{\mu}^2_I}{\hat{\mu}(0)(\hat{\lambda}_I + 2\hat{\mu}_I )}G_{11}^H \ma 0 & 0 \\ \\ 1 & 0 \am  \\
									&\phantom{=\;}-\frac{\hat{\mu}_I}{\hat{\mu}(0)} \ma 0 & 0 \\ \\ 1 & 0 \am \int_{0}^{H} V(y) \widetilde{G}_H(y)  dy  - \frac{\hat{\mu}_I}{2\hat{\mu}(0)} \ma 0 & 0 \\ \\ 1 & 0 \am \int_{0}^{H}  c_I y \mathcal{B}(y/2)  V(y) \widetilde{G}^H   dy,
								\end{align*}
								\begin{align*}
									\chi_0^{PS} =& - \frac{\hat{\mu}_I}{2 (\hat{\lambda}_I + 2 \hat{\mu}_I)} \ma   G_{11}^H  & 0 \\ \\   G_{12}^H & 0 \am  -\frac{1}{2} \widetilde{G}^H  - \frac{1}{2} \int_{0}^{H} V(y) \widetilde{G}_H(y)  dy \\
									& - \frac{1}{4} \int_{0}^{H} c_I y \mathcal{B}(y/2)  V(y) \widetilde{G}^H dy  + \ma -\theta_3 &\theta_2\\ \\ - \theta_1 & 0 \am     \widetilde{G}_H(0) - \frac{\hat{\mu}_I}{\omega^2} \frac{1}{2} \ma -\theta_3 &\theta_2\\ \\ - \theta_1 & 0 \am   \int_{0}^{H} V(y) \widetilde{G}^H dy, 
								\end{align*}
								and 
								\begin{align*}
									E(\xi) &= \; \xi^2 \frac{ \hat{\mu}^2_I}{ \omega^2 \hat{\mu}(0)}   \ma 0 & 0 \\ 1 & 0 \am \int_{0}^{H} e^{2y \xi}  V(y) \widetilde{G}^H dy  -\xi  \frac{\hat{\mu}_I}{2\omega^2} \int_{0}^{H} e^{2y \xi}  V(y) \widetilde{G}^H  dy \\
									&\phantom{=\;}  + \xi \frac{\hat{\mu}_I}{\hat{\mu}(0) } \ma 0 & 0 \\ 1 & 0 \am  \int_{0}^{H} e^{2y \xi}  V(y) \widetilde{ G}_H(y) dy - \frac{\xi}{2}  \frac{2 \hat{\mu}_I}{\hat{\mu}(0)} \ma 0 & 0 \\ 1 & 0 \am  \int_{0}^{H} e^{2y\xi} \frac{y}{2 \sigma} \mathfrak{d}(0,y)  V(y) \widetilde{G}^H \\
									&\phantom{=\;} + \frac{\hat{\mu}_I}{2\omega^2} \ma -\theta_3 & \theta_2 \\ -\theta_1 & 0 \am \int_{0}^{H} e^{2 y \xi}  V(y) \widetilde{G}^H dy - \frac{1}{2 } \int_{0}^{H} e^{2y \xi}  V(y) \widetilde{ G}_H(y) dy  + \frac{1}{2} \int_{0}^{H} e^{2 y \xi} \frac{y}{2 \sigma} \mathfrak{d}(0,y)  V(y) \widetilde{G}^H \\
									&\phantom{=\;}+ \frac{1}{2 \xi}  \ma -\theta_3 & \theta_2 \\ -\theta_1 & 0 \am \int_{0}^{H} e^{2y \xi} V(y) \widetilde{ G}_H(y) dy  - \frac{1}{2 \xi} \ma -\theta_3 & \theta_2 \\ -\theta_1 & 0 \am \int_{0}^{H} e^{2 y \xi} \frac{y}{2 \sigma} \mathfrak{d}(0,y)  V(y) \widetilde{G}^H.
								\end{align*}
								The determinant of $\mathcal{F}_{\Theta}(w_{PS}(\xi))$  is then equal to
								\begin{align*}
									\det \mathcal{F}_{\Theta}(w_{PS}(\xi)) = \xi^3 \left( \frac{\hat{\mu}_I}{\omega^2} c(0) + \mathscr{A}(\xi) \right)  + \xi^2 \mathscr{B}(\xi) + \mathscr{R}(\xi)
								\end{align*}
								where $\mathscr{A}(\xi)$ is
								\begin{align*}
									\mathscr{A}(\xi) := \frac{2\hat{\mu}_I^2}{\hat{\mu}(0) \omega^2} \int_{0}^{H} e^{2\xi y} V_{12}(y) dy.
								\end{align*}
								The part of the determinant obtained from the multiplication of the exponentially large terms with each other is included in the term $\mathscr{B}(\xi)$ and they are zero up to order $\xi^3$. In the case that also the term of order $\xi^2 \mathscr{B}(\xi)$ is zero, we get a worse estimate than the previous one. The term $\mathscr{B}(\xi)$ can be written as
								\begin{align*}
									\mathscr{B}(\xi) = C_1 \left( \int_{0}^{H} f_1(y) e^{2y\xi} dy \right) \left( \int_{0}^{H} f_2(t) e^{2t\xi} dt\right)
								\end{align*}
								where $f_1(y)$ and $f_2(y)$ are functions of the form $\sum_{i,j} C_{ij} y V_{ij}(y)$ which are in $L^1$ as $V \in \mathcal{V}_H$; hence
								\begin{align}\label{Estimate on B}
									|\mathscr{B}(\xi) | \leq C e^{4H|\xi|}.
								\end{align}
								The remainder $\mathscr{R}(\xi)$ contains also all the terms coming from the iterates $\mathcal{H}^{(l)}(x, \xi) $ defined in \eqref{Definition of Hl}, which are all dominated by the term $\xi^2 \mathscr{B}(\xi)$ as $\widetilde{ \mathcal{G}}(x,y)$ is of polynomial order $\xi^{-1}$, hence, by iteration, it keeps decreasing.
							\end{proof}
				
				We point out that the leading order term of the expansion in Lemma \ref{lemma det on unphys sheet} is exponentially increasing in $\xi$, while in (93) in \cite{Iantchenko2} it is only polynomially increasing. As explained in the introduction, the result in \cite{Iantchenko2} is therefore in strong contradiction with the definition of the unphysical sheet and resonances.

							In the following lemma we compute the asymptotic expansion of the Rayleigh determinant $\Delta (w_{P}(\xi)) $ and $\Delta (w_{S}(\xi)) $, from the asymptotic expansion of $\det \mathcal{F}_{\Theta}(w_{P}(\xi)) $ and  $\det \mathcal{F}_{\Theta}(w_{S}(\xi))  $.
							.

							\begin{lemma}\label{Lemma det P(xi)}
								Let $V \in \mathcal{V}_{H}$, then the Rayleigh determinants $\Delta (w_{P}(\xi)) $  and $\Delta (w_{S}(\xi)) $ for $\re \xi > 0$ and as $\re \xi \to \infty$ in $\Xi_{+,+}$ satisfies 
								\begin{equation}\label{delta di wp di xi}
									\Delta (w_{P}(\xi)) = \xi^4 \left( 8 \hat{\mu}_I^2   \left(G_{11}^H \right)^2 \right) + \xi^3 \left(\frac{2  \hat{\mu}(0) \omega^4}{\hat{\mu}_I}\right)  \mathscr{A}^P(\xi)+  \mathscr{R}^P(\xi),
								\end{equation}
								\begin{align}\label{delta di ws di xi}
									&\Delta (w_{S}(\xi))  = \xi^4 \left(   -8  \hat{\mu}_I^2 \left[G_{11}^H \right]^2 \right) +\xi^2 \left( \frac{2 \hat{\mu}(0)\omega^4}{\hat{\mu}_I}\right) \mathscr{A}^S(\xi) +  \mathscr{R}^S(\xi)
								\end{align}
								where
																\begin{align}\label{Definition of AP}
									\mathscr{A}^P(\xi) := \frac{2\hat{\mu}_I^3}{\hat{\mu}(0) \omega^4} G_{11}^H \int_{0}^{H} e^{2\xi y} a(y) dy,
								\end{align}
								\begin{align}\label{Definition of AS}
									\mathscr{A}^S(\xi) :=  \frac{\hat{\mu}_I^3}{\hat{\mu}(0) \omega^4} \int_{0}^{H} e^{2\xi y} \left[  \left( \theta_2 G_{12}^H - \frac{\omega^2}{\hat{\mu}_I} G_{21}(0) \right)  a(y) -  \theta_2 G_{11}^H b(y) \right] dy,
								\end{align}
								and 
								\begin{align*}
									&a(y) := V_{11}(y)  G_{11}^H + V_{12}(y) G_{12}^H \\
									&b(y) := V_{21}(y)  G_{11}^H + V_{22}(y) G_{12}^H.
								\end{align*}
								The remainder $\mathscr{R}^P(\xi)$ ($\mathscr{R}^S(\xi)$) contains the terms which are polynomially or exponentially smaller than $\mathscr{A}^P(\xi)$ ($\mathscr{A}^S(\xi)$).
							\end{lemma}
							\begin{proof}
								By \eqref{Connection B and F theta} we get
								\begin{align}\label{delta wp}
									\Delta (w_{P}(\xi)) &= \det A_1(\xi) \det \mathcal{F}_{\Theta}(w_{P}(\xi)) \det A_2(\xi) = (-2\ii \hat{\mu}(0) \hat{\mu}_I \xi)  \det \mathcal{F}_{\Theta}(w_{P}(\xi)) \left(\frac{\ii \omega^4}{\xi^2 \hat{\mu}_I^2}\right) \nonumber \\
									&= \left(\frac{2 \hat{\mu}(0) \omega^4  }{\hat{\mu}_I \xi}\right) \det \mathcal{F}_{\Theta}(w_{P}(\xi)).
								\end{align}
								As before, $\mathcal{F}_{\Theta}(w_{P}(\xi)) $ for $\xi \in \Xi_{+,+}$ is equal to  $\mathcal{F}_{\Theta}(\xi) $ for $\xi \in \Xi_{-,+}$. After some lengthy and tedious computation, as in the previous lemma, we obtain
								\begin{align*}
									\det \mathcal{F}_{\Theta}(w_{P}(\xi))  = \xi^5 \left( - \frac{4\hat{\mu}_I^3}{\omega^4 \hat{\mu}(0)} \left(G_{11}^H \right)^2 \right) + \xi^4 \mathscr{A}^P(\xi) + \mathscr{R}^P(\xi)
								\end{align*}
								with
								\begin{align*}
									\mathscr{A}^P(\xi) := \frac{2\hat{\mu}_I^3}{\hat{\mu}(0) \omega^4} G_{11}^H \int_{0}^{H} e^{2\xi y} \left[ G_{11}^H V_{11}(y) + G_{12}^H V_{12}(y) \right]  dy
								\end{align*}
								and the remainder $\mathscr{R}^P(\xi)$ containing all the other terms which are dominated by the term $\xi^4 \mathscr{A}^P(\xi)$. In the determinant, in contrast to the previous case, there are no terms containing the  product of two integrals containing exponentials, as all the exponentially large terms only appear in the first column. Substituting the result into \eqref{delta wp} we obtain 
								\begin{align}
									\Delta (w_{P}(\xi)) =& \xi^4 \left( -8 \hat{\mu}_I^2   \left(G_{11}^H \right)^2 \right) + \xi^3 \left(\frac{2  \hat{\mu}(0) \omega^4}{\hat{\mu}_I}\right)  \mathscr{A}^P({\xi})+  \mathscr{R}^P({\xi}).
								\end{align}
								Similarly, we obtain \eqref{delta di ws di xi}.
							\end{proof}
							
								We point out that the leading order terms of the expansion in \eqref{delta di wp di xi} and \eqref{delta di ws di xi} is exponentially increasing in $\xi$, while in (94) and (95) in \cite{Iantchenko2} they are only polynomially increasing. As explained in the introduction, the result in \cite{Iantchenko2} is therefore in stark contradiction with the definition of the unphysical sheet and resonances.

							Now, we use the results of Lemma \ref{lemma det on phys sheet}, Lemma \ref{lemma det on unphys sheet} and Lemma \ref{Lemma det P(xi)}, in order to obtain an asymptotic expansion of the entire function $F(\xi)$ defined in \eqref{equation entire funct F} in Section \ref{Section F entire}, as shown in the following theorem.

							\begin{theorem}\label{Lemma asymptotic of F}
								Let $V \in \mathcal{V}_{H}$, then the entire function $F(\xi)$, product of all Rayleigh determinants, admits the following form in the complex plane for $\re \xi > 0$ and as $\re \xi \to +\infty$:
								\begin{align*}
									F(\xi) &= \Delta (w_{S}(\xi)) \Delta (w_{P}(\xi))  \Delta (w_{PS}(\xi)) \Delta (\xi) \\
									&= \xi^{12} \left( 128 \hat{\mu}^2(0) c(0) \omega^{6} \hat{\mu}_I^3 \left[G_{11}^H\right]^4\right)\mathscr{A}(\xi)  - \xi^8 C\mathscr{B}(\xi) \mathscr{A}^P(\xi) \mathscr{A}^S(\xi) + \mathscr{R}(\xi).
								\end{align*}
							\end{theorem}
							\begin{proof}
								We know that $\det \mathcal{F}_{\Theta}(\xi) = \xi^3 \left( \frac{\hat{\mu}_I}{\omega^2} c(0) \right) + \mathcal{O}(|\xi|^{2})$ and using \eqref{Connection B and F theta} we have
								\begin{align*}
									\Delta (\xi) &= \det A_1(\xi) \det \mathcal{F}_{\Theta}(\xi) \det A_2(\xi) = (-2\ii \hat{\mu}(0) \hat{\mu}_I \xi)  \left(  \xi^3 \left( \frac{\hat{\mu}_I}{\omega^2} c(0) \right) + \mathcal{O}(|\xi|^{2}) \right) \left(\frac{\ii \omega^4}{\xi^2 \hat{\mu}_I^2}\right) \nonumber \\
									&= \xi^2 \left( 2 \hat{\mu}(0) c(0) \omega^2 \right) + \mathcal{O}(|\xi|)
								\end{align*}
								and
								\begin{align*}
									\Delta (w_{PS}(\xi))  &= \det A_1(\xi) \det \mathcal{F}_{\Theta}(w_{PS}(\xi)) \det A_2(\xi) \nonumber \\
									&=  -2\ii \hat{\mu}(0) \hat{\mu}_I \xi  \left(  \xi^3 \left( \frac{\hat{\mu}_I}{\omega^2} c(0) \right) + \xi^3 \mathscr{A}(\xi) + \xi^2 \mathscr{B}(\xi) + \mathscr{R}(\xi) \right) \left(\frac{\ii \omega^4}{\xi^2 \hat{\mu}_I^2}\right) \nonumber \\
									&= \xi^2 \left( 2  \hat{\mu}(0) c(0) \omega^2 \right) + \xi^2 \left(\frac{2 \omega^4 \hat{\mu}(0)}{\hat{\mu}_I }\right) \mathscr{A}(\xi)  + \xi \left(\frac{2 \omega^4 \hat{\mu}(0)}{\hat{\mu}_I }\right)\mathscr{B}(\xi) + \mathscr{R}(\xi).
								\end{align*}
								Then 
								\begin{align*}
									\Delta (w_{PS}(\xi)) \Delta (\xi)  =  \xi^4 \left[4  \hat{\mu}^2(0)c^2(0) \omega^4\right] + \xi^4 \left(\frac{2 \omega^6 \hat{\mu}^2(0) c(0)}{\hat{\mu}_I }\right) \mathscr{A}(\xi)  +  \xi^{3}\mathscr{B}(\xi)  + \mathscr{R}(\xi).
								\end{align*}
								From Lemma \ref{Lemma det P(xi)} we get
								\begin{align*}
									\Delta (w_{S}(\xi)) \Delta (w_{P}(\xi)) & = \xi^8 \left( 64 \hat{\mu}_I^4 \left[G_{11}^H\right]^4\right) - \xi^5 \left[\frac{4 \omega^8 \hat{\mu}^2(0)}{\hat{\mu}^2_I}\right] \mathscr{A}^P(\xi) \mathscr{A}^S(\xi) +\mathscr{R}(\xi)
								\end{align*}
								where we included all the other terms of $\Delta (w_{PS}(\xi)) \Delta (\xi) $ in the remainder $\mathscr{R}(\xi)$ as they are dominated by $\xi^5 \left[\frac{4 \omega^8 \hat{\mu}^2(0)}{\hat{\mu}^2_I}\right] \mathscr{A}^P(\xi) \mathscr{A}^S(\xi) $.
								Then, the entire function $F$, for $\re \xi \to +\infty$, is
								\begin{align}\label{expansion of F}
									 F(\xi) &= \Delta (w_{S}(\xi)) \Delta (w_{P}(\xi))  \Delta (w_{PS}(\xi)) \Delta (\xi) \nonumber \\
									&= \xi^{12} \left( 128 \hat{\mu}^2(0) c(0) \omega^{6} \hat{\mu}_I^3 \left[G_{11}^H\right]^4\right)\mathscr{A}(\xi)  - \xi^8 C\mathscr{B}(\xi) \mathscr{A}^P(\xi) \mathscr{A}^S(\xi) + \mathscr{R}(\xi).
								\end{align}
								In \eqref{expansion of F} we kept the largest polynomial order in $\xi$ and the largest term, which is the one obtained by the product of four exponentials. All the other terms of the remainder have smaller exponential order than the term containing $\xi^8 \mathscr{B}(\xi) \mathscr{A}^P(\xi)\mathscr{A}^S(\xi)$, or smaller polynomial order than this term, hence they are dominated by it.
							\end{proof}

							In Corollary \ref{Corollary Exp type of F} we found a first exponential type estimate of $F(\xi)$. In the next theorem we show an improved exponential type estimate of $F(\xi)$ after having computed the determinants of the Jost function in the different sheets. This make sense because in the computation of the determinants there have been several cancellations and the type of the exponentially large terms arising in the unphysical sheets is different from one sheet to another, as we have seen previously.
							
							\begin{theorem}\label{F is of expponential type}
								Let $V \in \mathcal{V}_{H}$, then the entire function $F(\xi)$ is of exponential type and for $\re \xi > 0$ and as $\re \xi \to \infty$ in the complex plane
								\begin{equation}\label{F is of exp type}
									|F(\xi)| \leq C \xi^{8} e^{8H|\re \xi|}
								\end{equation}
							\end{theorem}
							\begin{proof}
								The proof follows from Theorem \ref{Lemma asymptotic of F} as the second term dominates the first one and all the terms inside the remainder $\mathscr{R}(\xi)$. Moreover, we have seen in \eqref{Estimate on B} that
								\begin{align*}
									|\mathscr{B}(\xi)| \leq C e^{4H|\xi|}
								\end{align*}
								and from \eqref{Definition of AP} and \eqref{Definition of AS}, it holds
								\begin{align*}
									&|\mathscr{A}^P(\xi)| \leq C e^{2H|\xi|},
									&|\mathscr{A}^S(\xi)| \leq C e^{2H|\xi|},
								\end{align*}
								as $V \in \mathcal{V}_H \subset L^1$. Thus there exists a constant $C>0$ so that \eqref{F is of exp type} is satisfied.
							\end{proof}

							\section{Direct results}\label{Direct result Rayleigh section}
							
							In this section, we present the direct results on the number of resonances and the resonance-free regions, which are implied by the asymptotic expansion of $F(\xi)$ (see Theorem \ref{Lemma asymptotic of F}) and by the exponential type estimate of $F(\xi)$ (see Theorem \ref{F is of expponential type}). In Theorem \ref{F is entire} we proved that $F(\xi)$ is entire and together with the result of Theorem \ref{Lemma asymptotic of F}, we show in Theorem \ref{F is in Cartwright class} that $F(\xi)$ is in the Cartwright class (Definition \ref{Cartwright class definition}) with indices $\rho_{\pm}(F) \leq 8H$.

							\begin{theorem}\label{F is in Cartwright class}
								The function $F(z)$ is in the Cartwright class with 
								\begin{equation*}
									\rho_{\pm}(F) \leq 8H.
								\end{equation*}
							\end{theorem}
							\begin{remark}
								The exact value of the indices $\rho_{\pm}$ of the function $F$ might be obtained in the same way as in \cite{Sottile}. However, this would require solving the inverse resonance problem, which is beyond the scope of this work.
							\end{remark}
							\begin{proof}
								In order to prove that $F(z)$ is in the Cartwright class, we define $z:= \ii \xi= x + \ii y$, so $x=- \im \xi$ and $y=\re \xi$. We need to prove that
								\begin{equation*}
									\int_{\mathbb{R}}  \frac{\log^+ |F(x)| d x}{1 + x^2}<\infty, \quad \rho_+(F) \leq 8H, \; \rho_-(F) \leq 8 H, 
								\end{equation*}
								where $\rho_{\pm}(F) = \lim \sup_{y \to \infty} \frac{\log |F(\pm iy)|}{y}$.
								In Theorem \ref{Lemma asymptotic of F} we have seen that
								\begin{align*}
									F(\xi) &= -\xi^8 C \mathscr{B}(\xi) \mathscr{A}^P(\xi) \mathscr{A}^S(\xi) + \mathscr{R}(\xi) = -\xi^8 C \mathscr{B}(\xi) \mathscr{A}^P(\xi) \mathscr{A}^S(\xi) \left(1 + \frac{\mathscr{R}(\xi)}{ \xi^8 C \mathscr{B}(\xi) \mathscr{A}^P(\xi)}\right),
								\end{align*}
								where the last fraction that tends to zero as $\re \xi \to +\infty$.
								Then, we have 
								\begin{align*}
									\int_{\mathbb{R}}  \frac{\log^+ |F(x)| }{1 + x^2} dx\leq \int_{\mathbb{R}}  \frac{ \log \left(C|x|^{12}(1 + \mathcal{O}(\xi^{-1}))\right) }{1 + x^2} dx< \infty.
								\end{align*}
								For the index $\rho_{+}$ we have that $\xi = \ii (+\ii y) = - y$, so $\re \xi = -y$ and $|\re \xi |= y$, and thus
								\begin{align*}
									&\rho_+(F) \leq \lim \sup_{y \to \infty}  \frac{  8 \log |y| + 8H y }{y} =  8 H.
								\end{align*}
								While for the index $\rho_{-}$ we have that, $\xi = \ii (-\ii y) =  y$, so, $\re \xi = y$ and $|\re \xi |= y$, whence
								\begin{equation*}
									\rho_-(F) \leq \lim \sup_{y \to \infty}  \frac{  20 \log |y|  +  8 H y }{y} = 8H.
									\qedhere
								\end{equation*}
							\end{proof}
							
							The following result is an application of the Levinson theorem (see Theorem \ref{Levinson}) using that $F(\xi)$ is in the Cartwright class (see Theorem \ref{F is in Cartwright class}). 
							
							\begin{corollary}\label{Levinson theorem Rayleigh}
								Let $V \in \mathcal{V}_{H}$, then 
								\begin{align*}
									&\mathcal{N}_-(r,F) \leq \frac{8Hr}{\pi} (1 + \mathcal{O}(\xi^{-1})), \qquad r \to \infty, &	\mathcal{N}_+(r,F) \leq \frac{8Hr}{\pi} (1 + \mathcal{O}(\xi^{-1})), \qquad r \to \infty.
								\end{align*}
								Moreover, for each $\delta >0$ the number of complex resonances with real part with modulus $\leq r$ lying outside both of the two sectors $|\arg  \xi-\frac{\pi}{2}|<\delta$, $|\arg \xi - \frac{3\pi}{2}|<\delta$ is $o(r)$ for large $r$.
							\end{corollary}
							The previous result tells us how many resonances we are expected to have in a ball of radius $r$ in the complex plane, for large values of $r$.
							
							In the following theorem, we obtain some estimates of the resonances, which tell us where they are localized on the complex plane and vice-versa the forbidden domain for them. These are obtained from the asymptotics of $F(\xi)$ in Theorem \ref{Lemma asymptotic of F} and from the fact that the resonances are the zeros of $F(\xi)$.
							
							\begin{theorem}\label{Forbidden domain theorem Rayleigh}
								Let $V \in \mathcal{V}_{H}$, then for any zero $\xi_n$ of the function $F(\xi)$ the following estimate is fulfilled for $\re \xi_n \to \infty$:
								\begin{align*}
									|\xi_n|  \leq C e^{ 2H |\re \xi_n|}.
								\end{align*}
							\end{theorem}
							\begin{proof}
								From Theorem \ref{Lemma asymptotic of F} we know that the asymptotic  expansion of $F(\xi)$ is
								\begin{align*}
									\left|F(\xi) - \xi^{12} \left( 128 \hat{\mu}^2(0) c(0) \omega^{6} \hat{\mu}_I^3 \left[G_{11}^H\right]^4\right) \mathscr{A}(\xi) \right| 
									\leq | - \xi^8 C\mathscr{B}(\xi) \mathscr{A}^P(\xi) \mathscr{A}^S(\xi) + \mathscr{R}(\xi)| 
									\leq C |\xi^{8}| e^{8H|\re \xi|}.
								\end{align*}
								Evaluating this at a resonance $\xi=\xi_n$, as $\xi_n$ is a zero of $F$, we get
								\begin{align}\label{Inequality reson}
									| \xi_n^{12} | \left|\int_{0}^{H} e^{2\xi y} V_{12}(y) dy	\right| \leq C |\xi_n^{8}| e^{8H|\re \xi_n|},
								\end{align}
								since, we recall
								\begin{align*}
									\mathscr{A}(\xi) := \frac{2\hat{\mu}_I^2}{\hat{\mu}(0) \omega^2} \int_{0}^{H} e^{2\xi y} V_{12}(y) dy.
								\end{align*}
								The term $\left|\int_{0}^{H} e^{2\xi y} V_{12}(y) dy	\right|$ is bounded from below, because $V_{12}$ is  continuous and non zero in $(H-\epsilon, H)$ for $\epsilon>0$, hence with definite sign. Then the part of the integral in $(H-\epsilon, H)$ has no cancellation effects and dominate the rest.
								Thus \eqref{Inequality reson} becomes
								\begin{align*}
									| \xi_n |^4 \leq C  e^{8H|\re \xi_n|}
								\end{align*}
								and hence
								\begin{equation*}
									| \xi_n | \leq C  e^{2H|\re \xi_n|}.
									\qedhere
								\end{equation*}
							\end{proof}
							
							\begin{remark}
								We notice that the constant $C$ in Theorem \ref{Forbidden domain theorem Rayleigh} is not made explicit, because it is not so important. Indeed, $C$ would give a bound on the $\im \xi_n$ for small values of $\re \xi_n$, but since the estimate is obtained for large values of $\re \xi_n$, the constant $C$ is not important.
							\end{remark}

\bibliographystyle{plain}
\bibliography{MonographyCitations}

\end{document}